\documentclass[final, 12pt]{amsart}

\usepackage{amsmath, amssymb, esint}
\usepackage{graphicx, xcolor}
\usepackage{microtype}

\ifdefined\ifenablehyperref\else\expandafter\newif\csname
        ifenablehyperref\endcsname\enablehyperreftrue\fi
\ifenablehyperref\usepackage[final]{hyperref}\else\fi

\ifdefined\ifdraft\else\expandafter\newif\csname ifdraft\endcsname\fi
\ifdim\overfullrule>0pt \drafttrue\else\draftfalse\fi

\ifdefined\ifsharetheoremcounter\else\expandafter\newif\csname
        ifsharetheoremcounter\endcsname\sharetheoremcounterfalse\fi

\allowdisplaybreaks[4]

\newcounter{modifiedcounter}

\ifdraft
    \newcommand{\fixme}[1]{\marginpar{\parbox{0in}{\fbox{\parbox{0.6255in}{
            \raggedright \scriptsize #1}}}}}
    
    \newcommand{\dummyline}[1]{\vspace{0.5em}\noindent\rule{\textwidth}{#1 pt}
            \vspace{0.5em}}
    \newcommand{\dummynewpage}{\newpage}
    
\else
    \newcommand{\fixme}[1]{}
    
    \newcommand{\dummyline}[1]{}
    \newcommand{\dummynewpage}{}
    
\fi

\newlength{\defaultarrayrulewidth}
\newcommand{\owedge}{\mathbin{\ooalign{$\bigcirc$\cr\hidewidth$\,\,\wedge$
        \hidewidth\cr}}}
\newcommand{\ddpfir}[1]{\frac{\partial}{\partial#1}}

\newcommand{\Boxg}[1]{\left(\ddpfir{t}-\Delta_{#1}\right)}

\theoremstyle{remark}
\numberwithin{equation}{section}

\newtheorem*{claim*}{Claim}

\newtheorem{theorem}{Theorem}[section]

\ifsharetheoremcounter
    \newtheorem{definition}[theorem]{Definition}
    \newtheorem{claim}[theorem]{Claim}
    \newtheorem{lemma}[theorem]{Lemma}
    \newtheorem{proposition}[theorem]{Proposition}
    \newtheorem{corollary}[theorem]{Corollary}
    \newtheorem{remark}[theorem]{Remark}
\else
    \newtheorem{definition}{Definition}[section]
    \newtheorem{claim}{Claim}[section]
    \newtheorem{lemma}{Lemma}[section]
    \newtheorem{proposition}{Proposition}[section]
    \newtheorem{corollary}{Corollary}[section]
    \newtheorem{remark}{Remark}[section]
\fi

\DeclareMathOperator{\Inj}{inj}
\DeclareMathOperator{\id}{id}
\DeclareMathOperator{\interior}{int}
\DeclareMathOperator{\Ric}{Ric}
\DeclareMathOperator{\Rc}{Rc}
\DeclareMathOperator{\Rm}{Rm}

\DeclareMathOperator{\tr}{tr}
\DeclareMathOperator{\Vol}{Vol}
\DeclareMathOperator{\VolB}{VolB}

\DeclareMathOperator{\AVR}{AVR}

\def\x{\mathbf{x}}
\def\y{\mathbf{y}}
\def\z{\mathbf{z}}
\def\w{\mathbf{w}}

\begin{document}

\title{On an invariant curvature cone along $4$-dimensional Ricci flow}

\author[H. Ding]{Hongting Ding}
\address[Hongting Ding]{School of Science, Shenzhen Campus of Sun Yat-sen University, No. 66, Gongchang Road, Guangming District, Shenzhen, Guangdong 518107, P. R. China.}
\email{dinght@mail2.sysu.edu.cn}

\author[S. Huang]{Shaochuang Huang}
\address[Shaochuang Huang]{School of Science, Shenzhen Campus of Sun Yat-sen University, No. 66, Gongchang Road, Guangming District, Shenzhen, Guangdong 518107, P. R. China.}
\email{huangshch23@mail.sysu.edu.cn}
\thanks{S. Huang is partially supported by a regular fund from Shenzhen Science
        and Technology Program No. JCYJ20240813151005007 and a start-up
        fund from SYSU}

\author[Z. Peng]{Zhuo Peng}
\address[Zhuo Peng]{School of Science, Shenzhen Campus of Sun Yat-sen University, No. 66, Gongchang Road, Guangming District, Shenzhen, Guangdong 518107, P. R. China.}
\email{pengzh55@mail2.sysu.edu.cn}


\begin{abstract}
    In this paper, we study $4$-dimensional complete noncompact manifolds $(M,g)$ satisfying $\Rm_{g}\in\mathfrak{C}_{\eta,\mu}$ via Ricci flow. Under the additional assumption of maximal volume growth, we prove topological and geometric gap theorems. We also study $4$-dimensional complete  manifolds satisfying a lower bound with respect to $\mathfrak{C}_{\eta,\mu}$ and obtain regularity results for Gromov-Hausdorff limits of complete volume non-collapsed manifolds satisfying such curvature lower bounds.
    \end{abstract}

\maketitle

\section{Introduction and main results}

One of the core issues in  differential geometry is to study the topology and geometry of manifolds with certain curvature conditions. Ricci flow, introduced by Hamilton \cite{Hamilton82JDG} in 1982, has been proven to be a powerful tool for dealing with these kinds of problems. In this paper, we investigate the topology and geometry of $4$-dimensional complete Riemannian manifolds with a pointwise curvature condition via Ricci flow and we mainly discuss the noncompact case under a volume non-collapsed assumption.

In \cite{Hamilton86JDG,Hamilton97CAG}, Hamilton studies $4$-dimensional compact Riemannian manifolds with nonnegative curvature operator and nonnegative isotropic curvature respectively. In dimension $4$, the Lie algebra $\mathfrak{so}(4)$ splits as a direct sum of two copies of $\mathfrak{so}(3)$ and the space of
        $2$-forms $\wedge^2_p(M)$ admits the orthogonal decomposition
\[
    \wedge^2_p(M)=\wedge^+_p(M)\oplus\wedge^-_p(M),
\]
into the eigenspaces of the Hodge star operator $\star:\wedge^2_p(M)\to
        \wedge^2_p(M)$ of eigenvalues $\pm1$.
Then the curvature operator $\Rm$, viewed as a self-adjoint linear operator,
        admits a block decomposition into four pieces,
\begin{equation*} 
    \Rm=\begin{pmatrix}
        A & B \\
        B^T & C
    \end{pmatrix},
\end{equation*}
where $A:\wedge^+_p(M)\to\wedge^+_p(M)$ and $C:\wedge^-_p(M)\to\wedge^-_p(M)$ are linear transformations, $B:\wedge^-_p(M)\to\wedge^+_p(M)$ is a linear operator, and $B^T$ denotes the adjoint of $B$. We will review and discuss the decomposition of curvature operators in dimension $4$ with more details in Section \ref{sect decomp. 4mfd.}.

Let $\mathfrak{C}_{\eta,\mu}$ be the curvature cone
        defined by
\[
    \mathfrak{C}_{\eta,\mu}:=\left\{\Rm\in S_B^2(\mathfrak{so}(4))\left|\,
            \begin{aligned}
        &(B_2+B_3)^2\le\eta(A_1+A_2)(C_1+C_2), \\
        &A_2+A_3\le\mu(A_1+A_2), \\
        &C_2+C_3\le\mu(C_1+C_2)
    \end{aligned}\right.\right\},
\]
where $\eta,\mu$ are constants satisfying $\mu-1\ge\eta\ge0$ and $\mu>1$. Here $\{A_i,B_i,C_i\}_{i=1}^3$ are eigenvalues (singular values) of the linear operators $A,B,C$ respectively and satisfy $A_1\le A_2\le A_3, C_1\le C_2\le C_3$ and $0\le B_1\le B_2\le B_3$.
As we will discuss in Section \ref{pre-cone}, $\mathfrak{C}_{\eta,\mu}$ is
        a well-defined pointwise curvature cone since it is invariant under change of local orientation and 
         manifolds we consider in this paper are possibly non-oriented.

This curvature condition is stronger than  nonnegative isotropic curvature, and if we assume $\eta\leq1$ more, this curvature condition also implies nonnegative Ricci curvature. Recently, the study of topology and geometry of complete noncompact manifolds with weakly PIC1 is well developed, see for example \cite{HL21JGA,LT25GT,DSS24arXiv,DLSST26arXiv} and references therein. However, it seems not comparable between $\mathfrak{C}_{\eta,\mu}$ and $\mathfrak{C}_\mathrm{WPIC1}$,  where $\mathfrak{C}_\mathrm{WPIC1}$ is the cone of curvature operators with weakly PIC1. It is interesting to see whether one can obtain corresponding results for $4$-dimensional complete noncompact manifolds $(M,g)$ with $\Rm_{g}\in\mathfrak{C}_{\eta,\mu}$.  

The curvature cone $\mathfrak{C}_{\eta,\mu}$ has been considered by Hamilton in \cite{Hamilton97CAG}, as the intermediate process of proving a pinching theorem \cite[Theorem 1.1, Section 2]{Hamilton97CAG} for  $4$-dimensional compact Ricci flow with positive isotropic curvature. In particular, he proves that the curvature cone $\mathfrak{C}_{\eta,\mu}$ is preserved by Ricci flow on compact manifolds, see \cite[Theorem 1.3 and Theorem 1.4, Section 2]{Hamilton97CAG}. Then he applies these to understand the formation of singularities for $4$-dimensional compact Ricci flow with positive isotropic curvature, see also \cite{CZ06JDG2,CTZ12}. In this paper, we study the curvature cone $\mathfrak{C}_{\eta,\mu}$ on complete noncompact manifolds along $4$-dimensional Ricci flow. We first obtain the following short-time existence result for Ricci flow locally. 

\begin{theorem}[Theorem \ref{thm local exist.0}] \label{thm1-int}
    For any $v_0>0$, $\eta\in(0,1]$ and $\mu\ge\eta+1$, there exist
            $T(\eta,\mu,v_0)>0$ and $D_1(\mu,v_0)>0$ such that
            the following holds.
    Let $(M,g_0)$ be a $4$-dimensional Riemannian manifold.
    Suppose $B_{g_0}(p,s_0)\subset\subset M$ for some $p\in M$ and $s_0\ge4$
            such that
    \[\left\{\;\begin{aligned}
        &\Rm_{g_0}\in\mathfrak{C}_{\eta,\mu}\text{ on }B_{g_0}(p,s_0); \\
        &\VolB_{g_0}(x,1)\ge v_0\text{ for all }x\in B_{g_0}(p,s_0-1).
    \end{aligned}\right.\]
    Then there exists a Ricci flow $g(t)$ on $B_{g_0}(p,s_0-2)$
            for $t\in[0,T]$ with $g(0)=g_0$ such that
    \[
        \sup\limits_{B_{g_0}(p,s_0-2)}|\Rm|_{g(t)}\le\frac{D_1}{t}\text{ for all }t\in(0,T].
    \]
\end{theorem}

Then by a limiting argument, we obtain a short-time existence result for Ricci flow globally starting from a Riemannian metric possibly with unbounded curvature, see Corollary \ref{cor global exist.0}. With the help of this Ricci flow, we can show the topology of $4$-dimensional complete noncompact Riemannian manifold with this curvature condition when $\eta\in[0,1]$ is trivial if it is also of maximal volume growth.

\begin{theorem}[Theorem \ref{thm diff. R4}] \label{app1-int}
    Let $(M,g_0)$ be a complete noncompact $4$-dimensional Riemannian manifold with
            $\Rm_{g_0}\in\mathfrak{C}_{\eta,\mu}$ for some $\eta\in[0,1]$ and
            $\mu\ge\eta+1$ with $\mu>1$.
    If $(M,g_0)$ has maximal volume growth, then $M$ is diffeomorphic to
            $\mathbb{R}^4$.
\end{theorem}

Note that a complete noncompact Riemannian manifold with nonnegative Ricci curvature and  maximal volume growth may have non-trivial topology, for example Eguchi-Hanson metric. It is interesting to see how to strengthen the curvature condition such that the topology of a complete noncompact manifold is trivial. This is one of the motivations of this work. Another motivation is on the geometric gap theorem. Inspired by the work \cite{CLP24arXiv} by Chan, Lee and Peachey, we observe that if a Ricci flow coming out of a metric cone has curvature operator in $\mathfrak{C}_{\eta,\mu}$ and maximal volume growth, then it must be an expanding Ricci soliton. On the other hand, if $\eta\in[0,\frac{9}{16})$, we observe that the metric is Ricci pinched. Combining this with the expander structure, we obtain the following   geometric gap theorem.

\begin{theorem}[Corollary \ref{cor flat}] \label{app2-int}
    Suppose $(M,g_0)$ is a  $4$-dimensional complete noncompact Riemannian manifold with maximal volume growth and $\Rm_{g_0}\in\mathfrak{C}_{\eta,\mu}$ for some
            $\eta\in[0,\frac{9}{16})$ and $\mu\ge\eta+1$ with $\mu>1$.
    Then $(M,g_0)$ is isometric to flat Euclidean space.
\end{theorem}

Next, inspired by the work \cite{BCW19Inv} by Bamler, Cabezas-Rivas and Wilking, we consider manifolds with a lower bound with respect to $\mathfrak{C}_{\eta,\mu}$. We obtain the key differential inequality for the lower bound $l$ with respect to $\mathfrak{C}_{\eta,\mu}$, see Lemma \ref{lma evol. l}. With this key lemma, following the argument in \cite{BCW19Inv}, we obtain the following pseudo-locality theorem for the curvature cone $\mathfrak{C}_{\eta,\mu}$. 

\begin{theorem}[Theorem \ref{thm bcw thm1}]
    For any $v_0>0$, $\alpha_0\in[0,1]$, $\eta\in(0,1]$ and $\mu\ge\eta+1$,
            there exist $\hat{T}(\mu,v_0)>0$ and $D_1(\mu,v_0)>0$ such that the
            following holds.
    Let $(M,g(t)),t\in[0,T)$ be a smooth $4$-dimensional complete Ricci flow
            with \emph{bounded curvature} satisfying
    \[
        \VolB_{g(0)}(p,1)\ge v_0\text{ for all }p\in M\text{ and }
                \Rm_{g(0)}+\alpha_0\mathcal{I}_{g(0)}\in\mathfrak{C}_{\eta,\mu}.
    \]
    Then for all $t\in(0,T\wedge\hat{T}]$, we have
    \[
        \Rm_{g(t)}+D_1\alpha_0\mathcal{I}_{g(t)}\in\mathfrak{C}_{\eta,\mu}
                \text{ and }\sup\limits_{M}|\Rm|_{g(t)}\le\frac{D_1}{t}.
    \]
\end{theorem}

Then we are able to obtain two geometric applications as in \cite{BCW19Inv} for compact volume non-collapsed manifolds with a lower bound with respect to $\mathfrak{C}_{\eta,\mu}$ via the above Ricci flow pseudo-locality theorem, see Corollary \ref{cor BCW1} and Corollary \ref{cor BCW2}. 

Note that lower bound of $\mathfrak{C}_{\eta,\mu}$ may not imply lower bound of sectional curvature. However, we could apply the technique developed by Lai in \cite{Lai19AdvMath} to obtain the following short-time existence result for Ricci flow locally. 

\begin{theorem}[Theorem \ref{thm local exist.1}] \label{thm2-int}
    For any $\alpha_0\in(0,1]$, $v_0>0$, $\eta\in(0,1]$ and $\mu\ge\eta+1$,
            there exist $T(\eta,\mu,\alpha_0,v_0)>0$, $D_1(\mu,v_0)>0$ and
            $D_2(\eta,\mu,v_0)>0$ such that the following holds.
    Let $(M,g_0)$ be a $4$-dimensional Riemannian manifold.
    Suppose $B_{g_0}(p,s_0)\subset\subset M$ for some $p\in M$ and $s_0\ge4$
            such that
    \[\left\{\;\begin{aligned}
        &\Rm_{g_0}+\alpha_0\mathcal{I}_{g_0}\in\mathfrak{C}_{\eta,\mu}
                \text{ on }B_{g_0}(p,s_0); \\
        &\VolB_{g_0}(x,1)\ge v_0\text{ for all }x\in B_{g_0}(p,s_0-1).
    \end{aligned}\right.\]
    Then there exists a Ricci flow $g(t)$ on $B_{g_0}(p,s_0-2)$
            for $t\in[0,T]$ with $g(0)=g_0$ such that
    \[\left\{\;\begin{aligned}
        &\sup\limits_{B_{g_0}(p,s_0-2)}|\Rm|_{g(t)}\le\frac{D_1}{t}\text{ for all }t\in(0,T]; \\
        &\Rm_{g(t)}+D_2\alpha_0\mathcal{I}_{g(t)}\in\mathfrak{C}_{\eta,\mu}.
    \end{aligned}\right.\]
\end{theorem}

\begin{remark}
    Before we give applications from the above short-time existence result for local Ricci flow, we point out that Theorem \ref{thm2-int} covers Theorem \ref{thm1-int} by taking $\alpha_0=\frac{1}{2}$. However, since the existence time $T$ in Theorem \ref{thm2-int} depends on $\alpha_0$, we are not able to obtain the preservation of the curvature cone $\mathfrak{C}_{\eta,\mu}$ from Theorem \ref{thm2-int} directly. In fact, we still need Corollary \ref{cor preseve.} to obtain the topological and geometric gap theorems, Theorem \ref{app1-int} and Theorem \ref{app2-int}. On the other hand, the proof of Theorem \ref{thm2-int} also relies deeply on the curvature estimates in Section \ref{cur-est}, although the proofs of Theorem \ref{thm1-int} and Theorem \ref{thm2-int} are quite different. 
\end{remark}

Similarly, applying a limiting argument to Theorem \ref{thm2-int}, we obtain a short-time existence result for Ricci flow globally starting from a Riemannian metric possibly with unbounded curvature, see Corollary \ref{cor glabol exist.1}. With the help of this Ricci flow, we obtain the following regularity result for Gromov-Hausdorff limits of  complete noncompact volume non-collapsed manifolds with a lower bound with respect to $\mathfrak{C}_{\eta,\mu}$. 

\begin{corollary}[Corollary \ref{cor pGH Holder}]
    Suppose $\alpha_0\in(0,1]$, $v_0>0$, $\eta\in(0,1]$ and $\mu\ge\eta+1$.
    Let $(M_i,g_i)$ be a sequence of $4$-dimensional complete noncompact Riemannian
            manifolds such that for all $i$,
    \[\left\{\;\begin{aligned}
        &\Rm_{g_i}+\alpha_0\mathcal{I}_{g_i}\in\mathfrak{C}_{\eta,\mu}; \\
        &\VolB_{g_i}(x,1)\ge v_0\text{ for all }x\in M_i.
    \end{aligned}\right.\]
    Then there exist a smooth manifold $M$, a point $x_\infty\in M$ and
            a continuous distance metric $d_0$ on $M$ such that for some
            points $x_i\in M_i$, a subsequence of $(M_i,d_{g_i},x_i)$ converges
            to $(M,d_0,x_\infty)$ in pointed Gromov-Hausdorff sense.
    Furthermore, the metric space $(M,d_0)$ is locally bi-H\"older homeomorphic to the
            smooth manifold $M$ equipped with any smooth metric.
\end{corollary}

The paper is organized as follows. In Section 2, we review and discuss decomposition of curvature operators in dimension 4. In Section 3, we prove local and global preservation results of the curvature cone $\mathfrak{C}_{\eta,\mu}$. In Section 4, we prove some curvature estimates which will be used in Section 5 to construct Ricci flow starting from a metric $g_0$ satisfying $\Rm_{g_0}\in\mathfrak{C}_{\eta,\mu}$. Moreover, we provide topological and geometric applications in Section 5. In Section 6, we deal with short-time existence for Ricci flow with a lower bound with respect to $\mathfrak{C}_{\eta,\mu}$ and show geometric applications in this more general case. For readers' convenience, we include two appendices: One is a review of algebraic curvature operators, another is an existence result for cut-off functions we used in this paper. Throughout the paper, all manifolds are assumed to be smooth and connected, all Riemannian metrics are smooth, and all solutions to the Ricci flow are smooth solutions.

{\it Acknowledgements:} The authors would like to thank Pak-Yeung Chan and Man-Chun Lee for their interest in this work and for useful discussions.

\section{Decomposition of curvature operators in dimension 4} \label{sect decomp. 4mfd.}

In this section, we review and discuss the decomposition of curvature operators in dimension 4 following Hamilton \cite{Hamilton86JDG} in 1986. 

We start with dimension $\emph{3}$ as in \cite{Hamilton86JDG}.
Let $\{e_1,e_2,e_3\}$ be an orthonormal basis of $\mathbb{R}^3$ and
        $\{e_2\wedge e_3,e_3\wedge e_1,e_1\wedge e_2\}$ be the corresponding
        orthonormal basis of $\wedge^2\mathbb{R}^3$.
Then we can compute by \eqref{eqn Lie def.} the Lie structure constants
        $C_3^{12}=C_1^{23}=C_2^{31}=-1$.
If we denote components of the curvature operator $\Rm$ in the above basis by
\[
    \Rm_{\alpha\beta}=:\begin{pmatrix}
        a & b & c \\
        b & e & f \\
        c & f & k
    \end{pmatrix},
\]
then we can describe $\Rm^\#$ in the following simple way:
\[
    (\Rm^\#)_{\alpha\beta}=\begin{pmatrix}
        ek-f^2 & cf-bk & bf-ce \\
        cf-bk & ak-c^2 & bc-af \\
        bf-ce & bc-af & ae-b^2
    \end{pmatrix}.
\]

For any oriented Riemannian $\emph{4}$-manifold $(M,g)$, the space of
        $2$-forms $\wedge^2(M)$ admits the orthogonal decomposition
        \cite{Hamilton86JDG}
\[
    \wedge^2(M)=\wedge^+(M)\oplus\wedge^-(M),
\]
into the eigenspaces of the Hodge star operator $\star:\wedge^2(M)\to
        \wedge^2(M)$ of eigenvalues $\pm1$.
For any $p\in M$, after choosing an orthonormal frame, we may identify both
        $\wedge^2T_pM$ and $\wedge^2T_p^*M$ with $\wedge^2\mathbb{R}^4$
        and hence $\mathfrak{so}(4)$.
By the algebraic fact, see for example \cite[Lemma 9.1]{Fox21FMS}, that
\begin{equation} \label{eqn ad star}
    [u,\star v]=\star[u,v]=[\star u,v]\text{ for any }u,v\in\mathfrak{so}(4),
\end{equation}
one can check that this decomposition agrees with the Lie algebra
        decomposition $\mathfrak{so}(4)=\mathfrak{so}(3)\oplus\mathfrak{so}(3)$,
        which will significantly simplify the calculation of Lie structure
        constants if we take the basis of $\wedge_p^2(M)$ to be the union of
        bases of $\wedge_p^+(M)$ and $\wedge_p^-(M)$.
To this end, we pick any $3$-dimensional subspace $W\subset\mathbb{R}^4$ and
        a unit vector $e^\perp$ in the orthogonal complement of $W$, and then
        choose an orthonormal basis $\{w_1,w_2,w_3\}$ of
        $\wedge^2W\subset\wedge^2\mathbb{R}^4$ satisfying
\[
    [w_1,w_2]=w_3,\quad[w_2,w_3]=w_1,\quad[w_3,w_1]=w_2,
\]
for example one can take $w_1=e_3\wedge e_2, w_2=e_1\wedge e_3, w_3=e_2\wedge e_1$, where $\{e_1,e_2,e_3\}$ is an orthonormal basis of $W$. By the definition of Hodge star operator, we have $w_\alpha\wedge\star w_\alpha$
        equals to the volume form of $\mathbb{R}^4$ for each $\alpha=1,2,3$, and hence
        $\star w_\alpha$ is of the form $e^\perp\wedge w$ for some $w\in W$,
        which implies
\begin{equation} \label{eqn w dot star w}
    \langle w_\alpha,\star w_\beta\rangle=0\text{ for any }\alpha,\beta=1,2,3.
\end{equation}
Now we choose bases of $\wedge_p^+(M)$ and $\wedge_p^-(M)$ as follows:
\[\left\{\;\begin{aligned}
    \varphi_1^+&:=\frac{1}{\sqrt{2}}(w_1+\star w_1); \\
    \varphi_2^+&:=\frac{1}{\sqrt{2}}(w_2+\star w_2); \\
    \varphi_3^+&:=\frac{1}{\sqrt{2}}(w_3+\star w_3),
\end{aligned}\right.\quad\text{and}\quad\left\{\;\begin{aligned}
    \varphi_1^-&:=\frac{1}{\sqrt{2}}(w_1-\star w_1); \\
    \varphi_2^-&:=\frac{1}{\sqrt{2}}(w_2-\star w_2); \\
    \varphi_3^-&:=\frac{1}{\sqrt{2}}(w_3-\star w_3).
\end{aligned}\right.\]
Since $\star$ is self-adjoint, one can check by \eqref{eqn ad star} and
        \eqref{eqn w dot star w} that
        $\{\varphi_1^+,\varphi_2^+,\varphi_3^+\}$ and
        $\{\varphi_1^-,\varphi_2^-,\varphi_3^-\}$ form orthonormal bases of
        $\wedge_p^+(M)$ and $\wedge_p^-(M)$ respectively,
        and the Lie structure constants are given by
\[
    [\varphi_1^\pm,\varphi_2^\pm]=\sqrt{2}\varphi_3^\pm,\quad
            [\varphi_2^\pm,\varphi_3^\pm]=\sqrt{2}\varphi_1^\pm,\quad
            [\varphi_3^\pm,\varphi_1^\pm]=\sqrt{2}\varphi_2^\pm.
\]
In what follows, all computations are performed with respect to this basis. 

Now the curvature operator $\Rm$, viewed as a self-adjoint linear operator,
        admits a block decomposition into four pieces,
\begin{equation} \label{eqn decomp.4mfd.}
    \Rm=\begin{pmatrix}
        A & B \\
        B^T & C
    \end{pmatrix},
\end{equation}
where $A:\wedge_p^+(M)\to\wedge_p^+(M)$ and $C:\wedge_p^-(M)\to\wedge_p^-(M)$
        are linear transformations, $B:\wedge_p^-(M)\to\wedge_p^+(M)$
        is a linear operator, and $B^T$ denotes the adjoint of $B$.
From now on, we identify bilinear forms and linear operators 
        by the inner product and no longer distinguish them.
For any $u,v\in\wedge_p^2(M)$, writing $u=u^++u^-$ and $v=v^++v^-$ with
        $u^+,v^+\in\wedge_p^+(M)$ and $u^-,v^-\in\wedge_p^-(M)$, we have
\begin{equation}\begin{aligned} \label{eqn Rm u v}
    \Rm(u,v)&=\langle\Rm(u^+)+\Rm(u^-),v^+\rangle
            +\langle\Rm(u^+)+\Rm(u^-),v^-\rangle \\
    &=\langle A(u^+)+B(u^-),v^+\rangle+\langle B^T(u^+)+C(u^-),v^-\rangle \\
    &=A(u^+,v^+)+B(u^-,v^+)+B^T(u^+,v^-)+C(u^-,v^-).
\end{aligned}\end{equation}
By virtue of Lie algebra decomposition, we have
\begin{align*}
    (\Rm^\#)(u^+,v^+)&=\frac{1}{2}u_+^\alpha v_+^\beta C_\alpha^{\gamma\eta}
            C_\beta^{\delta\theta}\Rm_{\gamma\delta}\Rm_{\eta\theta} \\
    &=\frac{1}{2}u_+^\alpha v_+^\beta C_\alpha^{\gamma\eta}
            C_\beta^{\delta\theta}A_{\gamma\delta}A_{\eta\theta} \\
    &=:2(A^\#)(u^+,v^+),
\end{align*}
where $u^+=u_+^\alpha\varphi_\alpha^+$, etc.  Here we recognize $A$ as a $3$-dimensional linear operator and $A^\#$ as
        the corresponding $\#$-product operator defined in the $3$-dimensional case.
The factor $2$ comes from the product of Lie structure constants which are
        $\sqrt{2}$ in this case but $1$ in $3$-dimensional case.
Similarly, we have
\[
    (\Rm^\#)(u^-,v^-)=:2(C^\#)(u^-,v^-)\text{ and }
            (\Rm^\#)(u^-,v^+)=:2(B^\#)(u^-,v^+).
\]
Although $B$ is not a linear transformation on $\wedge_p^+(M)$ or $\wedge_p^-(M)$,
        we can still define the \emph{bilinear form} or \emph{linear operator}
        $B^\#$ from $\wedge_p^-(M)$ to $\wedge_p^+(M)$ by the same formula as
        \eqref{eqn sharp def.} thanks to the Lie algebra decomposition.
To be precise, for any $u^-=u^{\alpha^-}\varphi_\alpha^-\in\wedge_p^-(M)$,
        we can define $B^\#$ by
\[
    (B^\#u^-)_{\beta^+}:=(B^\#)_{\alpha^-\beta^+}u^{\alpha^-}\text{ where }
            (B^\#)_{\alpha^-\beta^+}:=\frac{1}{2}
            \left(\frac{C_{\alpha^-}^{\gamma\eta}}{\sqrt{2}}\right)
            \left(\frac{C_{\beta^+}^{\delta\theta}}{\sqrt{2}}\right)
            B_{\gamma\delta}B_{\eta\theta}.
\]
Therefore, we have the following block decomposition of $\Rm^\#$:
\[
    \Rm^\#=2\begin{pmatrix}
        A^\# & B^\# \\
        (B^T)^\# & C^\#
    \end{pmatrix}.
\]

We denote the eigenvalues of $A$ and $C$ by
\[
    A_1\le A_2\le A_3\quad\text{and}\quad C_1\le C_2\le C_3,
\]
respectively and the singular values of $B$ by
\[
    0\le B_1\le B_2\le B_3.
\]
We point out that $A_i$ is not the eigenvalue of $\Rm$ in general, but we still
        have for any unit eigenvector $\xi\in\wedge_p^+(M)$ of $A$ with eigenvalue
        $A_i$ that $\Rm(\xi,\xi)=A_i$ by \eqref{eqn Rm u v}.
Similar argument can be applied to $B_i$ and $C_i$.

By the decomposition of curvature operators \cite[(1.58)]{CLN06Book}
\[
    \Rm=\frac{\mathcal{R}}{12}\mathcal{I}+\frac{1}{2}\mathring{\Rc}\owedge g
            +\mathrm{Weyl},
\]
where $\mathcal{R}$ denotes the scalar curvature, $\mathring{\Rc}$ denotes the traceless Ricci tensor and
        $\mathrm{Weyl}$ denotes the Weyl curvature operator, and the
        algebraic fact, see for example \cite[Lemma 9.2]{Fox21FMS}, that
\[
    \star\circ\mathrm{Weyl}=\mathrm{Weyl}\circ\star,
\]
one can check that
\begin{equation} \label{eqn rc owedge g}
    \frac{1}{2}\mathring{\Rc}\owedge g=\begin{pmatrix}
        0 & B \\
        B^T & 0
    \end{pmatrix},
\end{equation}  see also \cite[Page 6]{CX23MathZ}. 
In particular, $B$ is identically zero if and only if the metric is Einstein  by the contracted second Bianchi identity. 
 We can also see that
\[
    \tr(A)=\tr(C)=\frac{\mathcal{R}}{4},
\] since the Weyl curvature operator is trace-free.
Moreover, we can calculate in orthonormal basis that
\begin{align*}
    |\mathring{\Rc}\owedge g|^2&=\sum_{i<j,\;k<l}
            \bigl((\mathring{\Rc}\owedge\id)_{ijkl}\bigr)^2 \\
    &=\frac{1}{4}\sum_{i,j,k,l}((\mathring{\Rc})_{ik}\delta_{jl}
            +(\mathring{\Rc})_{jl}\delta_{ik}
            -(\mathring{\Rc})_{il}\delta_{jk}
            -(\mathring{\Rc})_{jk}\delta_{il})^2 \\
    &=2|\mathring{\Rc}|^2+[\tr(\mathring{\Rc})]^2=2|\mathring{\Rc}|^2.
\end{align*}
It follows that
\begin{equation} \label{eqn norm Rc0}
    |\mathring{\Rc}|=2|B|,
\end{equation}
which is exactly \cite[(2.10) in Lemma 2.1]{CX23MathZ}, see also
        \cite[Section 2]{NiWa08} and \cite{LNW18}.

\dummyline{1}

\section{Preservation of the curvature cone} \label{pre-cone}

In this section, we will prove local and global preservation results of the curvature cone $\mathfrak{C}_{\eta,\mu}$. In order to simplify the proofs, we first set up as follows. 

Suppose $(M,g(t))$ is a solution to the $4$-dimensional Ricci flow for $t\in[0,T]$.
For each $(y,s)\in M\times[0,T]$, after choosing a local orientation, define
\[
    \Theta_{y,s}(M):=\left\{(\xi_1,\cdots,\xi_8)\left|\;\begin{aligned}
        &\xi_1,\cdots,\xi_4\in\wedge_{y,s}^+(M),\;\xi_5,\cdots,
                \xi_8\in\wedge_{y,s}^-(M) \\
        &\text{with }|\xi_{2i-1}|_{g(s)}=|\xi_{2i}|_{g(s)}=1\text{ and}\\
        &\langle\xi_{2i-1},\xi_{2i}\rangle_{g(s)}=0\text{ for }i=1,\cdots,4
    \end{aligned}\right.\right\}.
\]
Let $\mathcal{U}$ be an open neighborhood of $(x,t)$ such that all points in $\mathcal{U}$ are equipped with the same local orientation when one defines $ \Theta_{y,s}(M)$. Then we consider a  fiber bundle over $\mathcal{U}$ defined by

\[
    \Theta(\mathcal{U}):=\bigcup_{(y,s)\in\mathcal{U}}\Theta_{y,s}(M),
\]
and denote the natural projection by

\[
    \pi_\Theta:\Theta(\mathcal{U})\to\mathcal{U},\quad
            \pi_\Theta(\xi_1,\cdots,\xi_8)=(y,s).
\]

For any $\mathfrak{R}\in S_B^2(\mathfrak{so}(4))$ and $\xi=(\xi_1,\cdots,\xi_8)
        \in\Theta(\mathcal{U})$, let
\[\left\{\;\begin{aligned}
    X^{\mathfrak{R}}(\xi)&:=\mathfrak{R}(\xi_1,\xi_1)+\mathfrak{R}(\xi_2,\xi_2); \\
    W^{\mathfrak{R}}(\xi)&:=\mathfrak{R}(\xi_3,\xi_3)+\mathfrak{R}(\xi_4,\xi_4); \\
    Y^{\mathfrak{R}}(\xi)&:=\mathfrak{R}(\xi_5,\xi_5)+\mathfrak{R}(\xi_6,\xi_6); \\
    V^{\mathfrak{R}}(\xi)&:=\mathfrak{R}(\xi_7,\xi_7)+\mathfrak{R}(\xi_8,\xi_8); \\
    Z^{\mathfrak{R}}(\xi)&:=\mathfrak{R}(\xi_7,\xi_3)+\mathfrak{R}(\xi_8,\xi_4).
\end{aligned}\right.\]
These quantities define functions on $\Theta(\mathcal{U})$.

Let $\mathfrak{C}_{\eta,\mu}$ be the \emph{scaling invariant} curvature cone
        defined by
\[
    \mathfrak{C}_{\eta,\mu}:=\left\{\Rm\in S_B^2(\mathfrak{so}(4))\left|\,
            \begin{aligned}
        &(B_2+B_3)^2\le\eta(A_1+A_2)(C_1+C_2), \\
        &A_2+A_3\le\mu(A_1+A_2), \\
        &C_2+C_3\le\mu(C_1+C_2)
    \end{aligned}\right.\right\},
\]
where $\eta,\mu$ are constants satisfying $\mu-1\ge\eta\ge0$ and $\mu>1$.  Here $\{A_i,B_i,C_i\}_{i=1}^3$ are eigenvalues (singular values) of the linear operators $A,B,C$ respectively and satisfy $A_1\le A_2\le A_3, C_1\le C_2\le C_3$ and $0\le B_1\le B_2\le B_3$ as we explain in Section \ref{sect decomp. 4mfd.}.
Note that the definition of $\mathfrak{C}_{\eta,\mu}$ is independent of the
        choice of local orientation.
Indeed, if the orientation is reversed, then the Hodge star operator changes
        sign and hence $\wedge^+$ and $\wedge^-$ are interchanged.
Thus the block operators $A$ and $C$ are interchanged, and the block operator
        $B$ is replaced by its adjoint $B^T$.
Thanks to this symmetry of the definition of $\mathfrak{C}_{\eta,\mu}$, $\mathfrak{C}_{\eta,\mu}$ is a well-defined pointwise curvature cone for all
         $4$-dimensional Riemannian manifolds which is possibly non-oriented.

For any $(y,s)\in\mathcal{U}$, define

\begin{equation}\left\{\;\begin{aligned} \label{eqn def hatF}
    \hat{F}^1(y,s)&:=\inf_{\xi\in\Theta_{y,s}(M)}\eta X^{\Rm}(\xi)Y^{\Rm}(\xi)
            -Z^{\Rm}(\xi)^2; \\
    \hat{F}^2(y,s)&:=\inf_{\xi\in\Theta_{y,s}(M)}\mu X^{\Rm}(\xi)-W^{\Rm}(\xi); \\
    \hat{F}^3(y,s)&:=\inf_{\xi\in\Theta_{y,s}(M)}\mu Y^{\Rm}(\xi)-V^{\Rm}(\xi).
\end{aligned}\right.\end{equation}

Note that when $A_1+A_2\ge0$ and $C_1+C_2\ge0$, we have

\[
    \hat{F}^1(y,s)=\eta(A_1+A_2)(C_1+C_2)-(B_2+B_3)^2,
\]
and in any case,

\[\left\{\;\begin{aligned}
    \hat{F}^2(y,s)&=\mu(A_1+A_2)-(A_2+A_3); \\
    \hat{F}^3(y,s)&=\mu(C_1+C_2)-(C_2+C_3).
\end{aligned}\right.\]
Moreover, since $\mu>1$ and $\hat{F}^2\le(\mu-1)(A_1+A_2)$, we have
        $\hat{F}^2\ge0$ implies $A_1+A_2\ge0$, and similarly $\hat{F}^3\ge0$
        implies $C_1+C_2\ge0$.
By the above discussion, we can also rewrite the cone
\begin{equation} \label{eqn def.2 cone}
    \mathfrak{C}_{\eta,\mu}=\{\Rm\in S_B^2(\mathfrak{so}(4))\mid
            \hat{F}^1,\hat{F}^2,\hat{F}^3\ge0\}.
\end{equation}
By \cite[Theorem 1.2, Section 2]{Hamilton97CAG},
        the last two conditions of the cone guarantee that
        $\mathfrak{C}_{\eta,\mu}\subset\mathfrak{C}_\mathrm{WPIC}$, where
        $\mathfrak{C}_\mathrm{WPIC}$ is the cone of curvature operators with
        weakly positive isotropic curvature.

If we take $\mu=1$, which also forces $\eta=0$, then these three conditions of the
        cone reduce to $B=0$ and $A=C=kI$ for some $k\in\mathbb{R}$, but we lose
        the non-negativity of $A_1+A_2$ and $C_1+C_2$, which is crucial for our
        analysis.
For this reason, we define
\begin{equation} \label{eqn C01}
    \mathfrak{C}_{0,1^+}:=\bigcap_{\mu>1}\mathfrak{C}_{0,\mu}=\{\Rm\in
            S_B^2(\mathfrak{so}(4))\mid\Rm=k\mathcal{I}\text{ for some }k\ge0\}.
\end{equation}
One can readily check the second equality by the definition of the cone.

Next, we will show that the curvature cone $\mathfrak{C}_{\eta,\mu}$ satisfies the null-vector condition, which is basically \cite[Theorem 1.3 and Theorem 1.4, Section 2]{Hamilton97CAG}.

\begin{lemma} \label{lma null vector}
    Let $\mu-1\ge\eta\ge0$ and $\mu>1$. For any $\Rm\in\mathfrak{C}_{\eta,\mu}$
            and $\xi\in\Theta_{x,t}(M)$, we have the following properties:
    \begin{enumerate}
        \item[\textup{(\romannumeral1)}] If $X^{\Rm}(\xi)=0$ or $Y^{\Rm}(\xi)=0$,
                then $\Rm=0$ at $(x,t)$.
        \item[\textup{(\romannumeral2)}] If $\hat{F}^1(x,t)=
                \eta X^{\Rm}(\xi)Y^{\Rm}(\xi)-Z^{\Rm}(\xi)^2=0$,
            then at $(x,t)$,
            \[
                \eta Y^{\Rm}X^{Q(\Rm)}+\eta X^{\Rm}Y^{Q(\Rm)}
                        -2Z^{\Rm}Z^{Q(\Rm)}\ge0.
            \]
        \item[\textup{(\romannumeral3)}] If $\hat{F}^2(x,t)=
                \mu X^{\Rm}(\xi)-W^{\Rm}(\xi)=0$, then at $(x,t)$,
            \[
                \mu X^{Q(\Rm)}-W^{Q(\Rm)}\ge0.
            \]
            The same conclusion holds for $\hat{F}^3$.
    \end{enumerate}
\end{lemma}

\begin{proof}
    \textup{(\romannumeral1)}
    We only prove the case $X^{\Rm}(\xi)=0$ since the other case is similar.
    We already know that $X^{\Rm}(\xi)\ge A_1+A_2\ge0$, so $A_1+A_2=0$.
    By the first and second conditions of the cone, we have $B=0$ and $A_2+A_3=0$.
    Then $A_1=A_3$ and hence $A=0$.
    In addition, $C=0$ follows from $\tr(C)=\tr(A)=0$ and $C_1+C_2\ge0$.

    \textup{(\romannumeral2)}
    If $\eta=0$, then $B=0$, which implies $Z^{\Rm}(\xi)=0$ and hence the
            desired inequality holds trivially.
    We may therefore assume $\eta>0$.
    By changing the signs of $\xi_7$ and $\xi_8$ if necessary,
            we may assume $Z^{\Rm}(\xi)\ge0$.
    Then by assumption, we have
    \[\left\{\;\begin{aligned}
        X^{\Rm}(\xi)&=A_1+A_2\ge0; \\
        Y^{\Rm}(\xi)&=C_1+C_2\ge0; \\
        Z^{\Rm}(\xi)&=B_2+B_3\ge0.
    \end{aligned}\right.\]
    From the proof of \cite[Lemma 6.1]{Hamilton86JDG}
            (see also \cite{Hamilton97CAG}),
            it holds that
    \[\left\{\;\begin{aligned}
        X^{Q(\Rm)}(\xi)&\ge A_1^2+A_2^2+2(A_1+A_2)A_3+2B_1^2 \\
        &\ge2(A_1+A_2)(A_3+B_1); \\
        Y^{Q(\Rm)}(\xi)&\ge C_1^2+C_2^2+2(C_1+C_2)C_3+2B_1^2 \\
        &\ge2(C_1+C_2)(C_3+B_1); \\
        Z^{Q(\Rm)}(\xi)&\le(B_2+B_3)A_3+(B_2+B_3)C_3+2(B_2+B_3)B_1 \\
        &=(B_2+B_3)(A_3+C_3+2B_1).
    \end{aligned}\right.\]
    We remark that the above results differ from \cite{Hamilton86JDG} by a
            factor $2$ (e.g. $A_1=a_1/2$) due to the different definition of
            the inner product on $\wedge^2$.

    Therefore, we obtain
    \begin{align*}
        &\eta Y^{\Rm}X^{Q(\Rm)}+\eta X^{\Rm}Y^{Q(\Rm)}-2Z^{\Rm}Z^{Q(\Rm)} \\
        &\qquad\ge2\eta(C_1+C_2)(A_1+A_2)(A_3+B_1) \\
                &\qquad\qquad+2\eta(A_1+A_2)(C_1+C_2)(C_3+B_1) \\
                &\qquad\qquad-2(B_2+B_3)^2(A_3+C_3+2B_1) \\
        &\qquad=0,
    \end{align*}
    where we have used $\eta(A_1+A_2)(C_1+C_2)=(B_2+B_3)^2$.

    \textup{(\romannumeral3)} We only prove the conclusion for $\hat{F}^2$ since
            the case for $\hat{F}^3$ is similar.
    By assumption, we have
    \[\left\{\;\begin{aligned}
        X^{\Rm}(\xi)&=A_1+A_2\ge0; \\
        W^{\Rm}(\xi)&=A_2+A_3\ge0.
    \end{aligned}\right.\]
    From the proof of \cite[Lemma 6.1]{Hamilton86JDG} and \cite[Theorem 1.4]
            {Hamilton97CAG}, it holds that
    \begin{align*}
        X^{Q(\Rm)}(\xi)&\ge A_1^2+A_2^2+2(A_1+A_2)A_3+2B_1^2 \\
        &\ge(A_1+A_2)(2A_3+A_1), \\
    \end{align*}
    and
    \begin{align*}
        W^{Q(\Rm)}(\xi)&\le (A_2+A_3)A_3+2(A_2+A_3)A_1+(B_2+B_3)^2 \\
        &\le(A_2+A_3)(A_3+2A_1)+\eta(A_1+A_2)(C_1+C_2) \\
        &\le(A_2+A_3)(A_3+2A_1+(\mu-1)(A_1+A_2)) \\
        &\le(A_2+A_3)(A_3+A_1-A_2+\mu(A_1+A_2)),
    \end{align*}
    where we have used $(B_2+B_3)^2\le\eta(A_1+A_2)(C_1+C_2)$, $\eta\le\mu-1$
            and $C_1+C_2\le\frac{2}{3}\tr(C)=\frac{2}{3}\tr(A)\le A_2+A_3$.

    Therefore, the desired inequality follows from $A_2+A_3=\mu(A_1+A_2)$.
\end{proof}

We will use several times the technique of extending vectors at a local extremum
        point to construct smooth barrier functions defined on a spacetime
        neighborhood of the extremum point.
We state it here for later use.
We can extend $\xi$ to a local section $\tilde{\xi}$ of
        $\Theta(\mathcal{U})$ near $\pi_\Theta(\xi)\in\mathcal{U}$ by parallel
        translation with respect to $g(t)$ and then extend it to spacetime
        with $\nabla_t\tilde{\xi}=0$, then we can define the following
        \emph{smooth} functions on some neighborhood of $\pi_\Theta(\xi)$ in
        $\mathcal{U}$ by
\[\left\{\;\begin{aligned}
    X_\xi^{\Rm}(y,s)&:=X^{\Rm}(\tilde{\xi}(y,s)); \\
    Y_\xi^{\Rm}(y,s)&:=Y^{\Rm}(\tilde{\xi}(y,s)); \\
    Z_\xi^{\Rm}(y,s)&:=Z^{\Rm}(\tilde{\xi}(y,s)); \\
    W_\xi^{\Rm}(y,s)&:=W^{\Rm}(\tilde{\xi}(y,s)); \\
    V_\xi^{\Rm}(y,s)&:=V^{\Rm}(\tilde{\xi}(y,s)),
\end{aligned}\right.\]
and
\[\left\{\;\begin{aligned}
    \tilde{F}_\xi^1(y,s)&:=\eta X_\xi^{\Rm}(y,s)Y_\xi^{\Rm}(y,s)
            -Z_\xi^{\Rm}(y,s)^2\ge\hat{F}^1(y,s); \\
    \tilde{F}_\xi^2(y,s)&:=\mu X_\xi^{\Rm}(y,s)-W_\xi^{\Rm}(y,s)
            \ge\hat{F}^2(y,s); \\
    \tilde{F}_\xi^3(y,s)&:=\mu Y_\xi^{\Rm}(y,s)-V_\xi^{\Rm}(y,s)
            \ge\hat{F}^3(y,s).
\end{aligned}\right.\]

\dummyline{1}

Next, we will show the following local preservation of the curvature cone $\mathfrak{C}_{\eta,\mu}$. The idea of the proof  is by now  standard, see for example \cite{HT18AJM,LT19,HL21JGA,LT22CJM}. 

\begin{theorem} \label{thm t power k}
    For any $\mu-1\ge\eta>0$, $a\ge3$, $\sigma>0$, $r>0$ and $k\in\mathbb{N}$,
            there exist an absolute constant $c_1>0$ and a $c_2(\eta)>0$ such that
            the following holds.
    Let $(M,g(t)),t\in[0,T],T\le1$, be a smooth solution to the $4$-dimensional Ricci flow.
    Suppose $B_{g(t)}(p,r+4\sigma)\subset\subset M$ for some $p\in M$ and all
            $t\in[0,T]$ such that
    \begin{enumerate}
        \item[\textup{(\romannumeral1)}] $\Rm_{g(0)}\in \mathfrak{C}_{\eta,\mu}$
                on $B_{g(0)}(p,r+4\sigma)$;
        \item[\textup{(\romannumeral2)}] $|\Rm|_{g(t)}\le\frac{a}{t}$ on
                $B_{g(t)}(p,r+4\sigma)$ for all $t\in(0,T]$.
    \end{enumerate}
    Then
    \[
        \Rm_{g(t)}+t^k\mathcal{I}_{g(t)}\in\mathfrak{C}_{\eta,\mu}
                \text{ on }B_{g(t)}(p,r),
    \]
    for all $t\le T\wedge c_1\sigma^2a^{-1}\wedge
            c_2\min\{\sigma^4,\sigma^{4-\frac{2}{\hat{k}+2}}\}\hat{k}^{-2}$,
            where $\hat{k}:=\max\{k,2a\}$.
    Here $\mathcal{I}_{g(t)}$ denotes the constant curvature operator with
            sectional curvature $1$.
\end{theorem}

\dummyline{0.5}

\begin{remark}
    We point out that the above theorem does not cover the case $\eta=0$.
    Although in this case $B=0$ and the associated metrics are Einstein as
            discussed in Section \ref{sect decomp. 4mfd.}, our method fails to
            handle it.
\end{remark}

\dummyline{0.5}

\begin{proof}[Proof of Theorem \ref{thm t power k}]
    Consider the curvature operator
    \[
        \Lambda(x,t):=\Phi(x,t)\Rm(x,t)+\varphi(t)\mathcal{I}_{g(t)},
    \]
    where $\Phi(\cdot,t)\ge0\in C_c^2(M)$ and $\varphi(t)>0$ are functions
            to be determined later.
    We are going to show that there exists a uniform time $\tilde{T}>0$ such that
            $\Lambda(t)\in\mathfrak{C}_{\eta,\mu}$ for all $t\in[0,\tilde{T}]$.

    Clearly $\Lambda(0)\in\interior(\mathfrak{C}_{\eta,\mu})$ and hence
            $\Lambda(t)\in\mathfrak{C}_{\eta,\mu}$ for small $t>0$.
    Let $t_1>0$ be the first time such that $\Lambda(t_1)$ hits
            the boundary of $\mathfrak{C}_{\eta,\mu}$ at some point $x_0\in M$,
            which means that $\Lambda(t)\in\mathfrak{C}_{\eta,\mu}$ for all
            $t\in[0,t_1]$ but at $(x_0,t_1)$, after choosing a local orientation,
            we have one of the following conditions holds
            (see \eqref{eqn def.2 cone}):
    \begin{enumerate}
        \item[\textup{$1^\circ$}] $\hat{F}^1(x_0,t_1)=0,
                \quad\hat{F}^2(x_0,t_1)\ge0,\quad\hat{F}^3(x_0,t_1)\ge0$;
        \item[\textup{$2^\circ$}] $\hat{F}^1(x_0,t_1)\ge0,
                \quad\hat{F}^2(x_0,t_1)=0,\quad\hat{F}^3(x_0,t_1)\ge0$;
        \item[\textup{$3^\circ$}] $\hat{F}^1(x_0,t_1)\ge0,
                \quad\hat{F}^2(x_0,t_1)\ge0,\quad\hat{F}^3(x_0,t_1)=0$.
    \end{enumerate}

\dummyline{0.5}

    For the first case, since we still have $\hat{F}^2,\hat{F}^3\ge0$,
            there exists $\xi\in\Theta_{x_0,t_1}(M)$ such that
    \begin{align*}
        0=\hat{F}^1(x_0,t_1)&=\eta X^{\Lambda}(\xi)Y^{\Lambda}(\xi)
                -Z^{\Lambda}(\xi)^2 \\
        &=\eta(A_1+A_2)(C_1+C_2)-(B_2+B_3)^2.
    \end{align*}
    Here we have decomposed $\Lambda$ as in \eqref{eqn decomp.4mfd.} and denoted
            the corresponding eigenvalues or singular values by $A_i,B_i,C_i$,
            $i=1,2,3$.
    Extend $\xi$ to a local section  and define the
            associated \emph{functions} as discussed above.
    Denote
    \[\left\{\;\begin{aligned}
        X&:=X_{\xi}^{\Lambda}=\Phi X_{\xi}^{\Rm}+2\varphi=:\Phi\x+2\varphi; \\
        Y&:=Y_{\xi}^{\Lambda}=\Phi Y_{\xi}^{\Rm}+2\varphi=:\Phi\y+2\varphi; \\
        Z&:=Z_{\xi}^{\Lambda}=\Phi Z_{\xi}^{\Rm}=:\Phi\z.
    \end{aligned}\right.\]
    The missing terms $2\varphi$ in the last equalities above are due to the
            definition of $Z^\Lambda$, which forces the contribution of
            $\varphi\mathcal{I}$ to vanish.
    Note that at $(x_0,t_1)$, we have
    \[\left\{\;\begin{aligned}
        X&=A_1+A_2\ge0; \\
        Y&=C_1+C_2\ge0; \\
        Z&=B_2+B_3\ge0.
    \end{aligned}\right.\]

    We may assume $X,Y,Z>0$ at $(x_0,t_1)$, since the vanishing of any of $X,Y,Z$ 
            at $(x_0,t_1)$ implies at least one of $X$ and $Y$ vanishes and
            hence $\Lambda(x_0,t_1)=0$ by Lemma \ref{lma null vector}
            \textup{(\romannumeral1)}, which means $\hat{F}^2=0$ and then
            we can handle this in the case $2^\circ$.
    Clearly, $\Phi(x_0,t_1)>0$ since $\varphi>0$, and the smooth function
            $\tilde{F}_{\xi}^1:=\eta XY-Z^2\ge\hat{F}^1$ attains its minimum $0$
            at $(x_0,t_1)$.
    Then at $(x_0,t_1)$, we have
    \begin{align}
        0&=2(\eta XY-Z^2) \notag \\
        &=\eta Y(\Phi\x+2\varphi)+\eta X(\Phi\y+2\varphi)-2Z(\Phi\z) \notag \\
        &=\Phi(\eta Y\x+\eta X\y-2Z\z)+2\eta(X+Y)\varphi, \label{eqn order 0} \\
        0&=\nabla^{g(t_1)}(\eta XY-Z^2) \notag \\
        &=\eta Y\nabla X+\eta X\nabla Y-2Z\nabla Z \notag \\
        &=\eta Y(\x\nabla\Phi+\Phi\nabla\x)+\eta X(\y\nabla\Phi+\Phi\nabla\y)
                -2Z(\z\nabla\Phi+\Phi\nabla\z) \notag \\
        &=\Phi(\eta Y\nabla\x+\eta X\nabla\y-2Z\nabla\z)
                +\nabla\Phi(\eta Y\x+\eta X\y-2Z\z) \notag \\
        &=\Phi(\eta Y\nabla\x+\eta X\nabla\y-2Z\nabla\z)
                -2\eta(X+Y)\varphi\frac{\nabla\Phi}{\Phi}, \label{eqn order 1}
    \end{align}
    and
    \begin{align}
        0&\ge\Boxg{g(t_1)}(\eta XY-Z^2) \notag \\
        &=\eta Y\Box X+\eta X\Box Y-2Z\Box Z-2\eta\langle\nabla X,\nabla Y
                \rangle+2|\nabla Z|^2 \notag \\
        &=\eta Y(\x\Box\Phi+\Phi\Box\x-2\langle\nabla\Phi,\nabla\x
                \rangle+2\varphi')\notag \\
                &\qquad+\eta X(\y\Box\Phi+\Phi\Box\y-2\langle\nabla\Phi,\nabla\y
                \rangle+2\varphi') \notag \\
                &\qquad-2Z(\z\Box\Phi+\Phi\Box\z-2\langle\nabla\Phi,
                \nabla\z\rangle) \notag \\
                &\qquad+\frac{\eta|Y\nabla X+X\nabla Y|^2
                -4\eta XY\langle\nabla X,\nabla Y\rangle}{2XY} \notag \\
        &=\Box \Phi(\eta Y\x+\eta X\y-2Z\z)-2\langle\nabla\Phi,\eta Y\nabla\x
                +\eta X\nabla\y-2Z\nabla\z\rangle \notag \\
                &\qquad+\Phi(\eta Y\Box\x+\eta X\Box\y-2Z\Box\z)
                +2\eta(X+Y)\varphi' \notag \\
                &\qquad+\frac{\eta|Y\nabla X-X\nabla Y|^2}{2XY} \notag \\
        &\begin{aligned} \label{eqn order 2}
            &\ge-2\eta(X+Y)\varphi\frac{\Box\Phi}{\Phi}
                -4\eta(X+Y)\varphi\frac{|\nabla\Phi|^2}{\Phi^2}
                +2\eta(X+Y)\varphi' \\
                &\qquad+2\Phi(\eta YX_\xi^{Q(\Rm)}+\eta XY_\xi^{Q(\Rm)}
                -2ZZ_\xi^{Q(\Rm)}),
        \end{aligned}
    \end{align}
    where we denote $\left(\frac{\partial}{\partial t}-\Delta\right)$
            by $\Box$ for notational convenience and have used
            \eqref{eqn order 0}, \eqref{eqn order 1} and \eqref{eqn evol. Rm}.

    The last term in \eqref{eqn order 2} is the difficult one to handle.
    We first compute
    \begin{align*}
        Q(\Lambda)&=(\Phi\Rm+\varphi\mathcal{I})^2+(\Phi\Rm
                +\varphi\mathcal{I})^\# \\
        &=\Phi^2\Rm^2+2\varphi\Phi\Rm+\varphi^2\mathcal{I}+\Phi^2\Rm^\#
                +2\varphi\Phi(\Rm\#\mathcal{I})+\varphi^2\mathcal{I}^\# \\
        &=\Phi^2Q(\Rm)+\varphi\Phi\Ric\owedge\id+3\varphi^2\mathcal{I},
    \end{align*}
    where we have used Lemma \ref{lma op. sharp} in the last equality.
    For any $\zeta\in\wedge^+(M)\text{ or }\wedge^-(M)$ with $|\zeta|=1$,
            there exists a local orthonormal frame $\{e_i\}_{i=1}^4$ such that
            $\zeta=\frac{1}{\sqrt{2}}(e_1\wedge e_2\pm e_3\wedge e_4)$,
            see \cite[Section 2.2]{CX23MathZ} for example.
    It follows that
    \begin{align*}
        (\Ric\owedge\id)(\zeta,\zeta)&=\frac{1}{2}(\Ric\owedge\id)\big(
                e_1\wedge e_2\pm e_3\wedge e_4,
                e_1\wedge e_2\pm e_3\wedge e_4\big) \\
        &=\frac{1}{2}\big(\Ric(e_1,e_1)+\Ric(e_2,e_2)
                +\Ric(e_3,e_3)+\Ric(e_4,e_4)\big) \\
        &=\frac{1}{2}\mathcal{R}.
    \end{align*}
    Then we obtain
    \begin{align*}
        X_\xi^{Q(\Lambda)}(x_0,t_1)&=Q(\Lambda)(\xi_1,\xi_1)
                +Q(\Lambda)(\xi_2,\xi_2) \\
        &=\Phi^2Q(\Rm)(\xi_1,\xi_1)+\Phi^2Q(\Rm)(\xi_2,\xi_2)
                +\varphi\Phi\mathcal{R}+6\varphi^2 \\
        &=\Phi^2X_\xi^{Q(\Rm)}+(\Phi\mathcal{R}+6\varphi)\varphi,
    \end{align*}
    Similarly, we have
    \[
        Y_\xi^{Q(\Lambda)}(x_0,t_1)=\Phi^2Y_\xi^{Q(\Rm)}
                +(\Phi\mathcal{R}+6\varphi)\varphi.
    \]
    Next, we compute that
    \begin{align*}
        Z_\xi^{Q(\Lambda)}(x_0,t_1)&=Q(\Lambda)(\xi_7,\xi_3)
                +Q(\Lambda)(\xi_8,\xi_4) \\
        &=\Phi^2Z_\xi^{Q(\Rm)}
                +\varphi\Phi(\mathring{\Rc}\owedge\id)(\xi_7,\xi_3)
                +\varphi\Phi(\mathring{\Rc}\owedge\id)(\xi_8,\xi_4) \\
        &=\Phi^2Z_\xi^{Q(\Rm)}+2\varphi\Phi Z_\xi^{\Rm} \\
        &=\Phi^2Z_\xi^{Q(\Rm)}+2\varphi Z,
    \end{align*}
    where we have used \eqref{eqn rc owedge g} and \eqref{eqn Rm u v} in the
            third equality.

    On the other hand, by Lemma \ref{lma null vector} \textup{(\romannumeral2)},
            we have at $(x_0,t_1)$,
    \begin{equation}\begin{aligned} \label{eqn box comp.1}
        0&\le\eta YX_\xi^{Q(\Lambda)}+\eta XY_\xi^{Q(\Lambda)}
                -2ZZ_\xi^{Q(\Lambda)} \\
        &\le\Phi^2(\eta YX_\xi^{Q(\Rm)}+\eta XY_\xi^{Q(\Rm)}-2ZZ_\xi^{Q(\Rm)}) \\
                &\qquad+\eta(X+Y)(\Phi\mathcal{R}+6\varphi)\varphi.
    \end{aligned}\end{equation}
    Combining with \eqref{eqn order 2}, we arrive at
    \begin{equation} \label{eqn varphi prime}
        \varphi'\le\left(\frac{\Box\Phi}{\Phi}+2\frac{|\nabla\Phi|^2}{\Phi^2}
                +\mathcal{R}+\frac{6\varphi}{\Phi}\right)\varphi.
    \end{equation}

\dummyline{0.5}

    Since the case $2^\circ$ is simpler and the case $3^\circ$ is exactly
            the same as the case $2^\circ$, we only prove the case $2^\circ$
            briefly here.
    Suppose $\xi\in\Theta_{x_0,t_1}(M)$ satisfies $\mu X^\Lambda(\xi)
            -W^\Lambda(\xi)=0$ and we extend it to a local section.
    Denote
    \[\left\{\;\begin{aligned}
        X&:=X_{\xi}^{\Lambda}=\Phi X_{\xi}^{\Rm}+2\varphi=:\Phi\x+2\varphi; \\
        W&:=W_{\xi}^{\Lambda}=\Phi W_{\xi}^{\Rm}+2\varphi=:\Phi\w+2\varphi.
    \end{aligned}\right.\]
    Since $\mu>1$ and $\varphi>0$, we have $\Phi(x_0,t_1)>0$.
    Then at $(x_0,t_1)$, we have
    \begin{align}
        0&=\mu X-W=\Phi(\mu\x-\w)+2\varphi(\mu-1), \label{eqn case2 order 0} \\
        0&=\nabla^{g(t_1)}(\mu X-W)=\Phi(\mu\nabla\x-\nabla\w)
                +\nabla\Phi(\mu\x-\w) \notag \\
        &= \Phi(\mu\nabla\x-\nabla\w)
                -2(\mu-1)\varphi\frac{\nabla\Phi}{\Phi}, \notag
    \end{align}
    and
    \begin{align*}
        0&\ge\Boxg{g(t_1)}(\mu X-W) \\
        &\ge\Box\Phi(\mu\x-\w)-2\langle\nabla\Phi,\mu\nabla\x-\nabla\w\rangle
                +\Phi(\mu\Box\x-\Box\w)+2(\mu-1)\varphi' \\
        &\ge-2(\mu-1)\varphi\frac{\Box\Phi}{\Phi}-4(\mu-1)\varphi
                \frac{|\nabla\Phi|^2}{\Phi^2}+2(\mu-1)\varphi' \\
                &\qquad+2\Phi(\mu X_\xi^{Q(\Rm)}-W_\xi^{Q(\Rm)}).
    \end{align*}
    By Lemma \ref{lma null vector} \textup{(\romannumeral3)},
            we have at $(x_0,t_1)$,
    \begin{equation}\begin{aligned} \label{eqn box comp.2}
        0&\le\mu X_\xi^{Q(\Lambda)}-W_\xi^{Q(\Lambda)} \\
        &=\Phi^2(\mu X_\xi^{Q(\Rm)}-W_\xi^{Q(\Rm)})
                +(\mu-1)(\Phi\mathcal{R}+6\varphi)\varphi.
    \end{aligned}\end{equation}
    Combining the two inequalities above, we obtain \eqref{eqn varphi prime}.

\dummyline{0.5}

    Therefore, we have established \eqref{eqn varphi prime} in all cases.
    We claim that there exists constant $C_{\eta}\ge1$ such that at $(x_0,t_1)$,
    \begin{equation} \label{eqn varphi over Phi}
        |\mathcal{R}|+\frac{6\varphi}{\Phi}\le C_{\eta}|\Rm|_{g(t_1)}.
    \end{equation}

    Indeed, for the case $1^\circ$, we have $\eta XY=Z^2$ at $(x_0,t_1)$, which
            implies either $\sqrt{\eta}X\le Z$ or $\sqrt{\eta}Y\le Z$.
    We may assume $\sqrt{\eta}X\le Z$ since $\sqrt{\eta}Y\le Z$ is similar.
    Then it holds $\sqrt{\eta}(\Phi\x+2\varphi)\le\Phi\z$ and hence
            $\frac{\varphi}{\Phi}\le\frac{\z}{2\sqrt{\eta}}-\frac{\x}{2}\le
            (1+\frac{1}{\sqrt{\eta}})|\Rm|$.
    For the last two cases, we only show the case $2^\circ$.
    It suffices to note that by \eqref{eqn case2 order 0}, we have
            $\frac{\varphi}{\Phi}=-\frac{\mu\x-\w}{2(\mu-1)}
            \le\frac{\mu+1}{\mu-1}\cdot|\Rm|
            \le(1+\frac{2}{\eta})\cdot|\Rm|$.

    Plugging the claim into \eqref{eqn varphi prime}, we have proved that
            at $(x_0,t_1)$,
    \begin{equation} \label{eqn varphi prime 2}
        \varphi'\le\left(\frac{\Box\Phi}{\Phi}+2\frac{|\nabla\Phi|^2}{\Phi^2}
                +C_{\eta}|\Rm|_{g(t_1)}\right)\varphi.
    \end{equation}
    which is the same form as \cite[(2.2)]{HL21JGA}.
    The remaining proof is almost identical to that of
            \cite[Theorem 2.1]{HL21JGA}.
    We would like to provide a more detailed proof here for readers' convenience.

\dummyline{0.5}

    Let $\tilde{d}(\cdot,t):=d_{g(t)}(\cdot,p)+10\sqrt{at}$ and
            $\Phi(\cdot,t):=\phi(\tilde{d}(\cdot,t))$, where $\phi$ is
            some cut-off function defined on $[0,+\infty)$ to be chosen.
    We may assume $\Phi$ to be smooth when applying the maximum principle since we
            can use Calabi's trick, see \cite{HT18AJM} for example,  to take care
            of the case when the extremum occurs at a point in the cut locus of $p$.
    By applying \cite[Lemma 8.3 (a)]{Perelman02arXiv} with $K=\frac{a}{t}$ and
            $r_0=\sqrt{\frac{3t}{2a}}<\sqrt{t}$, we have
    \[
        \Boxg{g(t)}\tilde{d}\ge-6\sqrt{\frac{2a}{3t}}
                +5\sqrt{\frac{a}{t}}>0,
    \]
    in the sense of barrier, whenever $\tilde{d}\ge\sqrt{t}+10\sqrt{at}$.

    We first let $\phi$ be identical to 1 on $[0,r+2\sigma]$ and vanishing
            outside $[0,r+3\sigma]$.
    Then
    \begin{equation} \label{eqn box Phi}
        \Boxg{g(t)}\Phi=-\phi''+\phi'\Box\tilde{d}\le-\phi'',
    \end{equation}
    whenever $\tilde{d}\ge\sqrt{t}+10\sqrt{at}$ or $\tilde{d}\le r+2\sigma$.
    We may assume \eqref{eqn box Phi} holds at $(x_0,t_1)$, otherwise we have
            $r+2\sigma<\tilde{d}(x_0,t_1)<\sqrt{t_1}+10\sqrt{at_1}$, which
            means
    \begin{equation} \label{eqn t1 bound 1}
        t_1>\frac{(r+2\sigma)^2}{(1+10\sqrt{a})^2}\ge c_1\sigma^2a^{-1},
    \end{equation}
    where $c_1>0$ is an absolute constant.

\dummyline{0.5}

\begin{claim*}
    For any $\alpha>0$, there exists $\delta>0$ depending on $\alpha,\sigma$
            and the initial metric such that for all $t\in[0,\delta]$,
            $\tilde{d}\le r+2\sigma$,
    \[
        \Rm_{g(t)}+t^\alpha\mathcal{I}_{g(t)}\in\mathfrak{C}_{\eta,\mu}.
    \]
\end{claim*}

\begin{proof}[Proof of Claim]
    Since we are working on a compact subset of $M$, there exists $0<\rho<1/2$
            depending on $\sigma$ and the initial metric such that
            $|\Rm|_{g(t)}\le\rho^{-2}$ for all $t\in[0,\rho^2]$.
    Assume $t_1\le\rho^2$.
    Using \eqref{eqn varphi over Phi}, we have $\Phi\ge C_{\eta}^{-1}
            \rho^2\varphi$ at $(x_0,t_1)$.
    Let $\varepsilon:=\frac{1}{2(\alpha+1)}$.
    By Lemma \ref{fact exist. of cutoff}, we may choose $\phi$ satisfying
            $|\phi'|^2\le\frac{C_0}{4\varepsilon^2\sigma^2}\phi^{2-\varepsilon}$
            and
            $|\phi''|\le\frac{C_0}{2\varepsilon^2\sigma^2}\phi^{1-\varepsilon}$
            for some absolute constant $C_0\ge1$.
    Then \eqref{eqn varphi prime 2} becomes
    \[
        \varphi'\le\left(\frac{C_0}{\varepsilon^2\sigma^2}\phi^{-\varepsilon}
                +C_{\eta}\rho^{-2}\right)\varphi
                \le D_1(1+\varphi^\varepsilon)
                (\varepsilon\varphi^{\varepsilon-1})^{-1},
    \]
    where $D_1>0$ depends on $\varepsilon,\sigma,C_0,C_{\eta}$ and $\rho$.
    For any $0<s\le s_0:=\rho^2\wedge\frac{1}{32D_1^2}$, choose
            $\varphi=(t+s)^{\alpha+1}$.
    Then the inequality above implies at $t=t_1$,
    \[
        \left((t+s)^\frac{1}{2}\right)'=(\varphi^\varepsilon)'\le 2D_1.
    \]
    It follows that $t_1\ge s$.
    Hence $\Phi\Rm_{g(s)}+(2s)^{\alpha+1}\mathcal{I}_{g(s)}\in
            \mathfrak{C}_{\eta,\mu}$ for all $s\le s_0$.
    If we shrink $s_0$ if necessary to ensure $2^{\alpha+1}s_0\le1$, then
            $(2s)^{\alpha+1}\le2^{\alpha+1}s_0s^\alpha\le s^\alpha$,
            which completes the proof of the claim.
\end{proof}

\dummyline{0.5}

    Now we prove the theorem for $k\ge2C_{\eta} a\ge6$.
    In this case $\hat{k}=\max\{k,2a\}=k$.
    Choose $\varphi(t)=t^k$.
    Although $\varphi(0)=0$, by the claim above, we know that
            $\Lambda(t)\in\mathfrak{C}_{\eta,\mu}$ for small $t>0$.
    Let $\varepsilon:=\frac{1}{2k+3}>\frac{1}{3k}$ and $\phi$ be
            identical to $1$ on $[0,r+\sigma]$,
            vanishing outside $[0,r+2\sigma]$ and satisfying
            $|\phi'|^2\le\frac{C_0}{4\varepsilon^2\sigma^2}\phi^{2-\varepsilon}$
            and
            $|\phi''|\le\frac{C_0}{2\varepsilon^2\sigma^2}\phi^{1-\varepsilon}$, also  due to Lemma \ref{fact exist. of cutoff}.
    We may assume \eqref{eqn box Phi} holds at $(x_0,t_1)$, otherwise we obtain
            \eqref{eqn t1 bound 1} by the same argument.
    By \eqref{eqn varphi over Phi} and $|\Rm|_{g(t)}\le\frac{a}{t}$, we have
            $\Phi^{-1}\le C_{\eta}a\varphi^{-1}/t_1$.
    Then \eqref{eqn varphi prime 2} becomes
    \[
        \varphi'\le\frac{C_{\eta}a}{t_1}\varphi
                +\frac{C_0(C_{\eta}a)^\varepsilon}{\varepsilon^2
                \sigma^2t_1^\varepsilon}\varphi^{1-\varepsilon}.
    \]
    By the choice of $\varphi$, it follows that
    \begin{equation}\begin{aligned} \label{eqn t1 lower bound}
        t_1&\ge\left(\frac{\varepsilon^2\sigma^2(k-C_{\eta}a)}
                {C_0(C_{\eta}a)^\varepsilon}\right)^\frac{1}{1-\varepsilon(k+1)} \\
        &\ge\left(\frac{\sigma^2}{9C_0k^2}\left(\frac{k}{2}
                \right)^\frac{2k+2}{2k+3}\right)^\frac{2k+3}{k+2}
                \ge \frac{\sigma^{4-\frac{2}{k+2}}}{(18C_0)^2k^2}
                =\frac{\sigma^{4-\frac{2}{\hat{k}+2}}}{(18C_0)^2\hat{k}^2},
    \end{aligned}\end{equation}
    where we have used $C_{\eta}a\le k/2$, $C_0\ge1$ and $k\ge6$.
    Combining the definition of $\Phi$, we have established that for any
            $k\ge2C_{\eta}a$, it holds that
    \[
        \Rm_{g(t)}+t^k\mathcal{I}_{g(t)}\in\mathfrak{C}_{\eta,\mu}
                \text{ on }B_{g(t)}(p,r+\sigma-10\sqrt{at}),
    \]
    for all $t\le T\wedge c_1\sigma^2a^{-1}\wedge t_1$, where $t_1$ satisfies
            \eqref{eqn t1 lower bound}.

    For the case $1\le k<2C_{\eta}a$, we may assume $t_1\le1$.
    Then $C_{\eta}\hat{k}\ge2C_{\eta}a$
            and $t^k\ge t^{C_{\eta}\hat{k}}$ for all $t\le t_1$.
    By applying what we just proved to $C_{\eta}\hat{k}$, we obtain that
    \[
        \Rm_{g(t)}+t^{C_{\eta}\hat{k}}\mathcal{I}_{g(t)}
                \in\mathfrak{C}_{\eta,\mu}\text{ and hence }
                \Rm_{g(t)}+t^k\mathcal{I}_{g(t)}
                \in\mathfrak{C}_{\eta,\mu},
    \]
    for all $t\le T\wedge c_1\sigma^2a^{-1}\wedge t_1$, where $t_1$ satisfies

    \[
        t_1\ge\frac{\sigma^{4-\frac{2}{C_{\eta}\hat{k}+2}}}
                {(18C_0)^2(C_{\eta}\hat{k})^2}
                \ge\frac{\min\{\sigma^4,\sigma^{4-\frac{2}{\hat{k}+2}}\}}
                {(18C_0)^2C_{\eta}^2\hat{k}^2}.
    \]

\dummyline{0.5}

    Therefore, we have shown that
    \[
        \Rm_{g(t)}+t^k\mathcal{I}_{g(t)}\in\mathfrak{C}_{\eta,\mu}
                \text{ on }B_{g(t)}(p,r+\sigma-10\sqrt{at}),
    \]
    for all $t\le T\wedge c_1\sigma^2a^{-1}\wedge
            c_2\min\{\sigma^4,\sigma^{4-\frac{2}{\hat{k}+2}}\}\hat{k}^{-2}$,
            where $c_2>0$ is a constant only depending on $\eta$.
    Shrink $c_1$ if necessary to ensure $10\sqrt{at}\le\sigma$. We
            complete the proof.
\end{proof}

\dummyline{1}

Applying the above local preservation Theorem \ref{thm t power k}, we can obtain the following global preservation of the curvature cone $\mathfrak{C}_{\eta,\mu}$ and a uniqueness result for Ricci flow. 

\begin{corollary} \label{cor preseve.}
    Let $(M, g(t)),t\in[0,T]$ be a complete solution to the $4$-dimensional Ricci flow with
            $g(0)=g_0$ and
    \[
        |\Rm|(x,t)\le\frac{a(1+d_{g_0}(x,p)^2)}{t},
    \]
    for some $a\ge3$, $p\in M$ and for all $x\in M$.
    If $\Rm_{g_0}\in\mathfrak{C}_{\eta,\mu}$ for some $\mu-1\ge\eta\ge0$
            and $\mu>1$, then $\Rm_{g(t)}\in \mathfrak{C}_{\eta,\mu}$
            for all $t\in[0,T]$.
    In particular, if $\Rm_{g_0}$ is flat, then $\Rm_{g(t)}$ is flat for all
            $t\in[0,T]$.
\end{corollary}

\dummyline{1}

\begin{proof}  The proof adapts ideas in the proof of \cite[Corollary 2.1]{HL21JGA}.
Without loss of generality, we may assume $\eta>0$ since the case $\eta=0$
            can be handled by letting $\eta\rightarrow0$.
            Let $T_0=T\wedge\frac{1}{4a\beta^2}\le e^{-1}$, where $\beta=24\sqrt{2/3}$ is the
            constant in the shrinking balls lemma \cite[Corollary 3.3]{ST22JDG}.
    We claim that for any $\sigma\ge1$, we have
    \[
        B_{g(t)}(p,\sigma+4\sigma)\subset B_{g_0}(p,\sqrt{(11\sigma)^2-1})
                \text{ for }t\le T_0.
    \]
    Indeed, since we have $|\Rm|_{g(t)}\le\frac{a(11\sigma)^2}{t}$ in
            $B_{g_0}(p,\sqrt{(11\sigma)^2-1})$, by the shrinking balls lemma \cite[Corollary 3.3]{ST22JDG},
            it suffices to show that $5\sigma\le\sqrt{(11\sigma)^2-1}
            -\beta\sqrt{(11\sigma)^2at}$ for $t\le T_0$.
    This is true by our choice of $T_0\le\frac{1}{4a\beta^2}$ and $\sigma\ge1$.
    
    Now we can apply Theorem \ref{thm t power k} to $B_{g(t)}(p,\sigma+4\sigma)$
            to show that for any $k\in\mathbb{N}$, it holds that
    \[
        \Rm_{g(t)}+t^k\mathcal{I}_{g(t)}\in\mathfrak{C}_{\eta,\mu}
                \text{ on }B_{g(t)}(p,\sigma),
    \]
    for $t\le T_1\wedge c_2\sigma^{4-\frac{2}{\hat{k}+2}}\hat{k}^{-2}$,
            where $T_1:=T_0\wedge\frac{c_1}{(11^2a)}$
            and $\hat{k}:=\max\{2(11\sigma)^2a,k\}$.

    For any small $\varepsilon>0$, we can choose $\sigma\ge1$ satisfying
            $2(11\sigma)^2a=-\log\varepsilon$ and then take
            $k=\lfloor2(11\sigma)^2a\rfloor+1$.
    Then $t^k\le (e^{-1})^k\le(e^{-1})^{-\log\varepsilon}=\varepsilon$
            and hence
    \[
        \Rm_{g(t)}+\varepsilon\mathcal{I}_{g(t)}\in\mathfrak{C}_{\eta,\mu}
                \text{ on }B_{g(t)}(p,\sigma),
    \]
    for $t\le T_1\wedge c_2\sigma^{4-\frac{2}{k+2}}k^{-2}$.

    Letting $\varepsilon\to0$, we have $\sigma\to+\infty$ and
            $\sigma^{4-\frac{2}{k+2}}k^{-2}$ converges to some constant $C_a>0$,
            since $k=\lfloor2(11\sigma)^2a\rfloor+1$.
    Then $\Rm_{g(t)}\in\mathfrak{C}_{\eta,\mu}$ for all $t\in[0,T_2]$,
            where $T_2:=T_1\wedge c_2C_a/2$.
    If $T_2<T$, we can repeat the argument to show it holds for all $t\in[0,T]$.

    If $\Rm_{g_0}$ is flat, then $\Rm_{g_0}\in\mathfrak{C}_{\eta,\mu}$ for all
            $\mu-1\ge\eta\ge0$ and $\mu>1$, i.e. $\Rm_{g_0}\in\mathfrak{C}
            _{0,1^+}$, and hence $\Rm_{g(t)}\in\mathfrak{C}_{0,1^+}$.
    By \eqref{eqn C01} and Schur's lemma, we can write $\Rm_{g(t)}=\varphi(t)\mathcal{I}_{g(t)}$
            with $\varphi(0)=0$.
    By the continuity of $\varphi$, we have $\varphi$ is uniformly bounded on
            $[0,T]$.
    By the uniqueness of Ricci flow \cite[Theorem 1.1]{CZ06JDG}, it follows that
            $\varphi(t)\equiv0$ for all $t\in[0,T]$.
    This completes the proof.
\end{proof}

\dummyline{1}

\dummynewpage
\section{Curvature estimates} \label{cur-est}

We first recall a notion of curvature.
\begin{definition}[2-nonnegative flag curvature {\cite{BS10ALM}}]
        \label{def 2-non flag}
    For $n\ge3$, we say $(M^n,g)$ or $\Rm_{g}$ has 2-nonnegative flag
            curvature if
    \[
        R_{1313}+R_{2323}\ge0,
    \]
    for all orthonormal $3$-frames $\{e_1,e_2,e_3\}$.
\end{definition}

\dummyline{0.5}

\begin{remark}
    This is the case $\lambda=0$ in the characterization of
            $\mathfrak{C}_\mathrm{WPIC1}$ and a curvature operator with
            2-nonnegative flag curvature has nonnegative Ricci curvature. See \cite[Corollary 7.15]{Brendle10Book} for example. 
\end{remark}

\dummyline{1}

The following lemma shows the curvature cone $\mathfrak{C}_{\eta,\mu}$ implies 2-nonnegative flag curvature if $\eta\leq1$.

\begin{lemma} \label{lma 2-non flag}
    Curvature operators in $\mathfrak{C}_{\eta,\mu}$ with $\eta\in[0,1]$,
            $\mu\ge\eta+1$ and $\mu>1$ have 2-nonnegative flag curvature.
\end{lemma}

\dummyline{0.5}

\begin{proof}
 After choosing a local orientation, we can extend $\{e_1,e_2,e_3\}$ 
            to an orthonormal $4$-frame $\{e_1,e_2,e_3,e_4\}$ locally and
            choose a basis for $\wedge^+$ and $\wedge^-$
            as in \cite[Section 2.2]{CX23MathZ}.
Then we have
    \[\left\{\;\begin{aligned}
        R_{1313}&=\frac{1}{2}(A_{22}+C_{22}+2B_{22}); \\
        R_{2323}&=\frac{1}{2}(A_{33}+C_{33}-2B_{33}).
    \end{aligned}\right.\]
    It follows that
    \[
        R_{1313}+R_{2323}\ge\frac{1}{2}(A_1+A_2+C_1+C_2-2B_2-2B_3).
    \]
    The last two conditions in the definition of $\mathfrak{C}_{\eta,\mu}$ imply
            that $A_1+A_2\ge0$ and $C_1+C_2\ge0$ and hence the first condition
            with $\eta\le1$ implies
    \begin{align*}
        (B_2+B_3)^2&\le(A_1+A_2)(C_1+C_2) \\
        &\le\left(\frac{A_1+A_2+C_1+C_2}{2}\right)^2.
    \end{align*}
    This completes the proof.
\end{proof}

\dummyline{1}

As \cite[Lemma 3.1]{HL21JGA} and \cite[Lemma 2.1]{ST22JDG}, we will show the following local curvature estimate, which is an important ingredient for constructing local Ricci flow. 

\begin{lemma} \label{lma pseudo.}
    For any $v_0,K>0$, $\eta\in[0,1]$, $\mu\ge\eta+1$ and $\mu>1$, there exist
            $\bar{T}(\mu,v_0,K)$, $C_0(\mu,v_0,K)>0$ such that the following holds.
    Suppose $(M,g(t))$ is a $4$-dimensional Ricci flow for $t\in[0,T]$ and $p\in M$ such that
    $B_{g(t)}(p,r)\subset\subset M$ for all $t\in[0,T]$.
    Assume that on $B_{g(t)}(p,r),t\in[0,T]$,
    \[\left\{\;\begin{aligned}
        &\Rm_{g(t)}+Kr^{-2}\mathcal{I}_{g(t)}\in\mathfrak{C}_{\eta,\mu}; \\
        &\VolB_{g(0)}(p,r)\ge v_0r^4.
    \end{aligned}\right.\]
    Then for all $t\in(0,T]\cap(0,\bar{T}r^2]$,
    \[\left\{\;\begin{aligned}
        &|\Rm|_{g(t)}\le\frac{C_0}{t}
                &\quad\text{on }B_{g(t)}\left(p,\frac{r}{8}\right); \\
        &\Inj_{g(t)}\ge\sqrt{C_0^{-1}t}
                &\text{on }B_{g(t)}\left(p,\frac{r}{8}\right).
    \end{aligned}\right.\]
\end{lemma}

\dummyline{1}

Before we prove Lemma \ref{lma pseudo.}, we state two lemmata without proof.  

\begin{lemma}[{\cite[Lemma 2.3]{ST22JDG}}] \label{fact lower vol. ctrl.}
    Suppose $(M^n,g(t))$ is a Ricci flow for $t\in[0,T)$, such that
            $B_{g(t)}(x_0,\gamma)\subset\subset M$ for some $x_0\in M$ and
            $\gamma>0$ for all $t\in[0,T)$.
    Assume that
    \begin{enumerate}
        \item[\textup{(\romannumeral1)}] $\Ric_{g(t)}\ge-K$ on $B_{g(t)}(x_0,
                \gamma)$ for some $K>0$ and all $t\in[0,T)$;
        \item[\textup{(\romannumeral2)}] $|\Rm|_{g(t)}\le\frac{c_0}{t}$ on
                $B_{g(t)}(x_0,\gamma)$ for some $c_0<+\infty$ and all
                $t\in(0,T)$;
        \item[\textup{(\romannumeral3)}] $\VolB_{g(0)}(x_0,\gamma)
                \ge v_0$ for some $v_0>0$.
    \end{enumerate}
    Then there exist $\varepsilon_0(v_0,K,\gamma,n)>0$ and $\hat{T}(v_0,K,
            \gamma,c_0,n)>0$ such that
    \[
        \VolB_{g(t)}(x_0,\gamma)\ge\varepsilon_0
                \text{ for all }t\in[0,T)\cap[0,\hat{T}).
    \]
\end{lemma}

\dummyline{1}

\begin{lemma}[{\cite[Lemma 5.1]{ST22JDG}}] \label{fact decay or not}
    For $n\ge2$, take $\beta(n)>0$ as in the shrinking balls lemma \cite[Corollary 3.3]{ST22JDG}.
    Suppose $(M^n,g(t))$ is a Ricci flow for $t\in[0,T]$, $x_0\in M$, $r_0>0$
            and $c_0>0$ such that $B_{g(t)}(x_0,r_0)\subset\subset M$ for all
            $t\in[0,T]$.
    Then at least one of the following holds:
    \begin{enumerate}
        \item[\textup{(1)}] For each $t\in(0,T]$ with $t<\frac{r_0^2}{c_0\beta^2
                }$, we have
            \[
                B_{g(t)}(x_0,r_0-\beta\sqrt{c_0t})\subset B_{g(0)}(x_0,r_0)
            \]
            and
            \[
                |\Rm|_{g(t)}<\frac{c_0}{t}\text{ on }B_{g(t)}(
                        x_0,r_0-\beta\sqrt{c_0t}).
            \]
        \item[\textup{(2)}] There exist $\bar{t}\in(0,T]$ with $\bar{t}<
                \frac{r_0^2}{c_0\beta^2}$ and $\bar{x}\in B_{g(\bar{t})}(x_0,r_0
                -\frac{1}{2}\beta\sqrt{c_0\bar{t}})$ such that
            \[
                Q:=|\Rm|(\bar{x},\bar{t})\ge\frac{c_0}{\bar{t}}\text{ and }
                        |\Rm|(x,t)\le4Q,
            \]
            whenever $x\in B_{g(\bar{t})}(\bar{x},\frac{\beta c_0}{8}Q^{-\frac
                    {1}{2}})$ and $t\in[\bar{t}-\frac{1}{8}c_0Q^{-1},\bar{t}]$.
    \end{enumerate}
\end{lemma}

\dummyline{1}

We also need the following result for ancient solution to the Ricci flow with $\Rm_{g(t)}\in\mathfrak{C}_{\eta,\mu}$. 

\begin{lemma} \label{lma ancient sol.}
    Suppose $(M,g(t)),t\in(-\infty,0]$ is a $4$-dimensional complete non-flat ancient solution
            to the Ricci flow with bounded curvature and
            $\Rm_{g(t)}\in\mathfrak{C}_{\eta,\mu}$
            for some $\mu-1\ge\eta\ge0$ with $\mu>1$.
    Then $(M,g(t))$ has nonnegative curvature operator and hence
            the volume growth is non-Euclidean.
\end{lemma}

\dummyline{0.5}

\begin{proof}
    By \cite[Lemma 4.4]{CL23MathAnn}, \cite[Proposition 11.4]{Perelman02arXiv}
            (see also \cite[Theorem 1]{Ni05MRL} and \cite[Lemma 6.3.1]{CaoZhu}), it suffices to show that
            $(M,g(t))$ has uniformly PIC.
    By \cite[Definition 2.2 \textup{(\romannumeral3)}]{CL23MathAnn}, we only
            need to verify that there exists $\Lambda\ge1$ such that
    \begin{equation} \label{eqn unif.pic}
        0<\max\{A_3,B_3,C_3\}\le\Lambda\min\{A_1+A_2,C_1+C_2\}.
    \end{equation}

Note that \eqref{eqn unif.pic} is a well-defined pointwise curvature condition since it is invariant under change of orientation.  We first check the second inequality at any fixed point with a local orientation. Without loss of generality, we may assume $A_1+A_2\le C_1+C_2$.
    Recall that $\mathfrak{C}_{\eta,\mu}\subset\mathfrak{C}_\mathrm{WPIC}$
            and we have $A_1+A_2\ge0$, $C_1+C_2\ge0$ and
    \[
        C_1+C_2\le\frac{2}{3}\tr(C)=\frac{2}{3}\tr(A)\le(A_2+A_3)\le\mu(A_1+A_2).
    \]
    It follows that
    \[\left\{\;\begin{aligned}
        &A_3\le A_2+A_3\le\mu(A_1+A_2); \\
        &B_3\le B_2+B_3\le\sqrt{\eta(A_1+A_2)(C_1+C_2)}
                \le\sqrt{\eta\mu}(A_1+A_2); \\
        &C_3\le C_2+C_3\le\mu(C_1+C_2)\le\mu^2(A_1+A_2).
    \end{aligned}\right.\]
    We can take $\Lambda:=\max\{\mu,\sqrt{\eta\mu},\mu^2\}$. The above argument is basically contained in \cite[Corollary 1.5]{Hamilton97CAG}.

    It remains to check the first inequality in \eqref{eqn unif.pic}.
    For any $t_0<0$, there exists $x_0\in M$ such that $\Rm(x_0,t_0)\neq0$.
    By Lemma \ref{lma null vector} \textup{(\romannumeral1)}, we have at
            $(x_0,t_0)$, $A_1+A_2>0$, after choosing a local orientation.
    On the other hand, we already know from the proof of Lemma
            \ref{lma null vector} \textup{(\romannumeral3)} that
    \[
        \Boxg{g(t)}(A_1+A_2)\ge(A_1+A_2)(2A_3+A_1)\ge0,
    \]
    in the barrier sense.
    Note that $A_1+A_2$ may not be a well-defined function on the whole manifold
            when $M$ is non-oriented.
    However, since $M$ is assumed to be connected, every two points in $M$
            can be connected by a continuous curve and then we can find a neighborhood of 
            this curve such that  all points in this  neighborhood are equipped with the same local orientation when one considers  the decomposition of $\wedge^2(M)$ via Hodge star operator. 
Then we can  apply the strong maximum principle (e.g. \cite[Proposition 12.47]
            {CGIKLN08Book}) to conclude that $A_1+A_2>0$ and then $A_3>0$ everywhere
           with the compatible local orientation  for all $t>t_0$. Therefore, we have $\max\{A_3,B_3,C_3\}>0$ on $M\times(t_0,0]$.
    Since $t_0$ is arbitrary, we conclude that $\max\{A_3,B_3,C_3\}>0$
            on $M\times(-\infty,0]$.
    This completes the proof.
\end{proof}

\dummyline{1}

\begin{proof}[Proof of Lemma \ref{lma pseudo.}]
    The proof follows that of \cite[Lemma 2.1]{ST22JDG} and
            we use Lemma \ref{lma ancient sol.} to draw a contradiction. We provide details here for completeness.

    We may assume $r=1$ by parabolic rescaling.
    Since $\eta\le1$, by Lemma \ref{lma 2-non flag}, we have
    \[
        \Ric_{g(t)}\ge-3K\text{ on }B_{g(t)}(p,1),t\in[0,T].
    \]

    Suppose the curvature estimate in the lemma fails.
    Then for any $c_k\to\infty$ and for any $t_k\to0$ sufficiently small
            with $c_kt_k\to0$, we can find a sequence of $\eta_k\in[0,\min\{1,\mu-1\}]$,
            a sequence of Ricci flows $(M_k,\tilde{g}_k(t))$ for $t\in[0,t_k]$
            and a sequence of points $x_k\in M_k$ with
            $B_{\tilde{g}_k(t)}(x_k,1)\subset\subset M_k$
            for all $t\in[0,t_k]$, such that
    \[\left\{\;\begin{aligned}
        &\VolB_{\tilde{g}_k(0)}(x_k,1)\ge v_0; \\
        &\Rm_{\tilde{g}_k(t)}+K\mathcal{I}_{\tilde{g}_k(t)}
                \in\mathfrak{C}_{\eta_k,\mu}\text{ on }B_{\tilde{g}_k(t)}(x_k,1)
                \text{ for }t\in[0,t_k],
    \end{aligned}\right.\]
    and
    \[
        |\Rm|_{\tilde{g}_k(t)}<\frac{c_k}{t}\text{ on }
                B_{\tilde{g}_k(t)}(x_k,\frac{1}{8})\text{ for }t\in(0,t_k),
    \]
    but
    \begin{equation} \label{eqn contra.}
        |\Rm|_{\tilde{g}_k(t_k)}=\frac{c_k}{t_k}\text{ at some point in }
                \overline{B_{\tilde{g}_k(t_k)}(x_k,\frac{1}{8})}.
    \end{equation}

    By volume comparison and applying Lemma \ref{fact lower vol. ctrl.} to
            $B_{\tilde{g}_k}(x_k,\frac{1}{8})$, there exist
            $\varepsilon_0(v_0,K)>0$ and $\hat{T}(v_0,K,c_k)>0$ such that
    \begin{equation} \label{eqn vol. bound}
        \VolB_{\tilde{g}_k(t)}(x_k,1)\ge\varepsilon_0
                \text{ for all }t\in[0,t_k]\cap[0,\hat{T}).
    \end{equation}
    We may choose $t_k\le\hat{T}$ to ensure \eqref{eqn vol. bound}
            holds for all $t\in[0,t_k]$.

    Now we apply Lemma \ref{fact decay or not} to $\tilde{g}_k(t)$ with
            $r_0=\frac{1}{4}$ and $c_0=c_k$.
    Since we may further assume $t_k<\frac{r_0^2}{4c_k\beta^2}$,
            we know from \eqref{eqn contra.} that Assertion 1 in
            Lemma \ref{fact decay or not} can not hold.
    Then there exist $\bar{t}_k\in[0,t_k]$ and $\bar{x}_k\in B_{\tilde{g}_k(
            \bar{t}_k)}(x_k,r_0-\frac{1}{2}\beta\sqrt{c_k\bar{t}_k})$ such that
    \[
        Q_k:=|\Rm|_{\tilde{g}_k(\bar{t}_k)}(\bar{x}_k)\ge\frac{c_k}{\bar{t}_k}
                \to\infty\text{ and }|\Rm|_{\tilde{g}_k(t)}(x)\le4Q_k,
    \]
    whenever $x\in B_{\tilde{g}_k(\bar{t}_k)}(\bar{x}_k,\frac{\beta c_k}{8}
            Q_k^{-\frac{1}{2}})$ and $t\in[\bar{t}_k-\frac{1}{8}c_kQ_k^{-1},
            \bar{t}_k]$.

    By volume comparison, for all $r\in(0,\frac{1}{2})$,
    \begin{align*}
        \frac{\VolB_{\tilde{g}_k(\bar{t}_k)}(\bar{x}_k,r)}{\omega_4 r^4}
                &\ge\frac{\VolB_{\tilde{g}_k(\bar{t}_k)}(\bar{x}_k,r)}{
                \VolB_{-K}(r)}
                \ge\frac{\VolB_{\tilde{g}_k(\bar{t}_k)}(\bar{x}_k,\frac{1}{2})}
                {\VolB_{-K}(\frac{1}{2})} \\
        &\ge\frac{\VolB_{\tilde{g}_k(\bar{t}_k)}(x_k,\frac{1}{4})}{
                \VolB_{-K}(\frac{1}{4})}
                \cdot\frac{\VolB_{-K}(\frac{1}{4})}{\VolB_{-K}(\frac{1}{2})} \\
        &\ge\frac{\VolB_{\tilde{g}_k(\bar{t}_k)}(x_k,1)}{
                \VolB_{-K}(1)}
                \cdot\frac{\VolB_{-K}(\frac{1}{4})}{\VolB_{-K}(\frac{1}{2})}.
    \end{align*}
Hence by \eqref{eqn vol. bound}, there exists $\eta_1(v_0,K)>0$ such that
    \[
        \frac{\VolB_{\tilde{g}_k(\bar{t}_k)}(\bar{x}_k,r)}{r^4}\ge\eta_1.
    \]

    Let $g_k(t):=Q_k\tilde{g}_k(\bar{t}_k+Q_k^{-1}t)$ for $t\in[-\frac{1}{8}c_k,
            0]$.
    Then we have
    \[\left\{\;\begin{aligned}
        &|\Rm|_{g_k(0)}(\bar{x}_k)=1; \\
        &|\Rm|_{g_k(t)}(x)\le4\text{ on }B_{g_k(0)}(
                \bar{x}_k,\frac{\beta c_k}{8})\times[-\frac{1}{8}c_k,0]; \\
        &\Rm_{g_k(t)}+KQ_k^{-1}\mathcal{I}_{g_k(t)}\in\mathfrak{C}_{\eta_k,\mu},
    \end{aligned}\right.\]
    and for all $0<r<\frac{1}{2}\sqrt{Q_k}\to\infty$,
    \begin{equation} \label{eqn vol. scaled}
        \frac{\VolB_{g_k(0)}(\bar{x}_k,r)}{r^4}\ge\eta_1>0.
    \end{equation}

   After passing to a subsequence, we may assume
            $\eta_k\to\eta_\infty\in[0,\min\{1,\mu-1\}]$.
    By Cheeger-Gromov-Taylor \cite{CGT82JDG}, $\Inj_{g_k(0)}(\bar{x}_k)$
            is bounded from below uniformly.
    Hence we can apply Hamilton's compactness theorem to show that
            $(M_k,\allowbreak g_k(t),\bar{x}_k)$ admits a subsequence converging to a
            complete ancient solution $(M_\infty,g_\infty(t),x_\infty)$
            with bounded curvature and $|\Rm|_{g_\infty(0)}(x_\infty)=1$.
    By the uniqueness of Ricci flow \cite[Theorem 1.1]{CZ06JDG}, $g_\infty(t)$
            is non-flat for all $t\in(-\infty,0]$.
    Moreover, we have $\Rm_{g_\infty(t)}\in\mathfrak{C}_{\eta_\infty,\mu}$
            and \eqref{eqn vol. scaled} passes to the limit,
            which contradicts Lemma \ref{lma ancient sol.}.

    With the curvature estimate established, the injectivity radius
            estimate follows from Lemma \ref{fact lower vol. ctrl.} and
            Cheeger-Gromov-Taylor \cite{CGT82JDG}.
    This completes the proof.
\end{proof}

\dummyline{1}



\dummynewpage
\section{Existence of Ricci flow I and applications}

In this section, we will prove existence theorems for Ricci flow starting from a volume non-collapsed metric with its curvature operator in $\mathfrak{C}_{\eta,\mu}$ which may not be bounded. Then we will show topological and geometric gap theorems via Ricci flows we construct.


We first state the following local existence proposition for Ricci flow which is based on a result by Hochard \cite[Corollaire IV.1.2]{Hochard19PhD} without proof.

\begin{proposition}[{\cite[Proposition 4.2]{LT25JDG}}] \label{fact local RF}
    Suppose $(N^n,h_0)$ is a smooth manifold (not necessarily complete)
            satisfying
    \[
        |\Rm|_{h_0}\le\rho^{-2}\text{ for some }\rho>0.
    \]
    Then there exist constants $\alpha(n)\in(0,1]$, $\Lambda(n)>0$ and a smooth
            Ricci flow $h(t)$ on $N$ for $t\in[0,\alpha\rho^2]$ with the
            properties that
    \begin{enumerate}
        \item[\textup{(\romannumeral1)}] $h(0)=h_0$ on $N^\rho_{h_0}:=
                \{x\in N\mid B_{h_0}(x,\rho)\subset\subset N\}$;
        \item[\textup{(\romannumeral2)}] $|\Rm|_{h(t)}\le\Lambda\rho^{-2}$
                throughout $N\times[0,\alpha\rho^2]$.
    \end{enumerate}
\end{proposition}

\dummyline{1}

\begin{theorem} \label{thm local exist.0}
    For any $v_0>0$, $\eta\in(0,1]$ and $\mu\ge\eta+1$, there exist
            $T(\eta,\mu,v_0)>0$ and $D_1(\mu,v_0)>0$ such that
            the following holds.
    Let $(M,g_0)$ be a $4$-dimensional Riemannian manifold.
    Suppose $B_{g_0}(p,s_0)\subset\subset M$ for some $p\in M$ and $s_0\ge4$
            such that
    \[\left\{\;\begin{aligned}
        &\Rm_{g_0}\in\mathfrak{C}_{\eta,\mu}\text{ on }B_{g_0}(p,s_0); \\
        &\VolB_{g_0}(x,1)\ge v_0\text{ for all }x\in B_{g_0}(p,s_0-1).
    \end{aligned}\right.\]
    Then there exists a Ricci flow $g(t)$ on $B_{g_0}(p,s_0-2)$
            for $t\in[0,T]$ with $g(0)=g_0$ such that
    \[
        \sup\limits_{B_{g_0}(p,s_0-2)}|\Rm|_{g(t)}\le\frac{D_1}{t}\text{ for all }t\in(0,T].
    \]
\end{theorem}

\dummyline{1}

\begin{proof} 
With preparations in Section \ref{pre-cone} and Section \ref{cur-est}, the idea of the proof is now clear, see for example \cite{Hochard19PhD,ST22JDG,ST21GT,HL21JGA}, and we provide details here for completeness.

    Choose $\rho>0$ small enough such that
    \[
        |\Rm|_{g_0}\le\rho^{-2}\text{ on }B_{g_0}(p,s_0).
    \]
    By Proposition \ref{fact local RF}, there exists a local Ricci flow
            $g(t)$ on $B_{g_0}(p,r_0)$, where $r_0:=s_0-1$, with $g(0)=g_0$
            for a small time interval with curvature bounded by some number
            depending on $\rho$.

    Let $\gamma(\eta,\mu,v_0),L(\eta,\mu,v_0)>0$ and $a(\mu,v_0)\ge3$ be constants
            to be determined later.
    Choose $t_0\le1$ small enough, which might depend on $g_0$, such that
    \[
        |\Rm|_{g(t)}\le\frac{a}{t}\text{ on }B_{g_0}(p,r_0)\times(0,t_0].
    \]

\begin{claim} \label{claim5.1}
    By shrinking $t_0$ if necessary, we have for any
            $x\in B_{g_0}(p,r_0-L\sqrt{t_0})$ and $t\in(0,t_0]$,
    \[
        |\Rm|(x,t)\le\frac{C_0}{t}\text{ and }\Inj_{g(t)}(x)\ge\sqrt{C_0^{-1}t},
    \]
    where $C_0(\mu,v_0)>0$ is the constant in Lemma \ref{lma pseudo.} when $K=1$.
\end{claim}

\begin{proof}[Proof of the Claim \ref{claim5.1}]
    Assume $L\ge\gamma+\beta\sqrt{a}$, where $\beta$ is the constant
            in the shrinking balls lemma \cite[Corollary 3.3]{ST22JDG}.
    Let $x\in B_{g_0}(p,r_0-L\sqrt{t_0})$ and $t\in[0,t_0]$.
    By the shrinking balls lemma \cite[Corollary 3.3]{ST22JDG},
    \[
        B_{g(t)}(x,\gamma\sqrt{t_0})\subset B_{g_0}(x,\gamma\sqrt{t_0}+\beta
                \sqrt{at})\subset B_{g_0}(x,L\sqrt{t_0})\subset B_{g_0}(p,r_0).
    \]

    Assume $\gamma>4$ and let $\delta_1(\mu,v_0)\in[1,\frac{\gamma}{4}]$
            be a constant to be determined later.
    Take
    \[
        \tilde{g}(t):=(\delta_1^2t_0)^{-1}g(\delta_1^2t_0t),
                t\in[0,\delta_1^{-2}].
    \]
    Then
    \[
        B_{\tilde{g}(t)}(x,1+\frac{\gamma}{4\delta_1})
                \subset\subset B_{\tilde{g}(t)}(x,\frac{\gamma}{\delta_1})
                =B_{g(\delta_1^2t_0t)}(x,\gamma\sqrt{t_0})\subset B_{g_0}(p,r_0).
    \]
    Apply Theorem \ref{thm t power k} on $B_{\tilde{g}(t)}(x,1+\frac{\gamma}
            {4\delta_1})$ with $\sigma=\frac{\gamma}{16\delta_1}\ge\frac{1}{4}$.
    Since $\hat{k}\ge6$ in Theorem \ref{thm t power k}, we have
            $\sigma^{4-\frac{2}{\hat{k}+2}}
            \ge(4\sigma)^\frac{15}{4}(\frac{1}{4})^4$.
    Then there exist constants $c_1,c_2(\eta)>0$ such that
            on $B_{\tilde{g}(t)}(x,1)$,
    \[
        \Rm_{\tilde{g}(t)}+t\mathcal{I}_{\tilde{g}(t)}\in\mathfrak{C}_{\eta,\mu}\text{ for all }t\le
                \delta_1^{-2}\wedge\frac{c_1\gamma^2a^{-1}}{16^2\delta_1^2}
                \wedge\frac{c_2\gamma^\frac{15}{4}a^{-2}}{16^4\delta_1^4}.
    \]
    Scaling back, it arrives that on $B_{g(t)}(x,\delta_1\sqrt{t_0})$,
    \[
        \Rm_{g(t)}+\frac{t}{\delta_1^4t_0^2}\mathcal{I}_{g(t)}\in\mathfrak{C}_{\eta,\mu}
                \text{ for all }t\le t_0\wedge\frac{c_1\gamma^2a^{-1}}{16^2}t_0
                \wedge\frac{c_2\gamma^\frac{15}{4}a^{-2}}{16^4\delta_1^2}t_0.
    \]

    Assume $\delta_1\sqrt{t_0}\le1$.
    Since $\Ric_{g_0}\ge0$, by volume comparison,
    \[
        \frac{\VolB_{g_0}(x,\delta_1\sqrt{t_0})}{(\delta_1\sqrt{t_0})^4}
                \ge\VolB_{g_0}(x,1)\ge v_0.
    \]
    We may apply Lemma \ref{lma pseudo.} on $B_{g(t)}(x,\delta_1\sqrt{t_0})$
            with $r=\delta_1\sqrt{t_0}$ and $K=1$ since $\frac{1}{\delta_1^2t_0}
            \ge\frac{t}{\delta_1^4t_0^2}$ for all $t\le t_0$.
    Then there exist $\bar{T}(\mu,v_0),C_0(\mu,v_0)>0$ such that on
            $B_{g_0}(p,r_0-L\sqrt{t_0})$,
    \[
        |\Rm|_{g(t)}\le\frac{C_0}{t}
                \text{ and }\Inj_{g(t)}(x)\ge\sqrt{C_0^{-1}t}
                \text{ for all }t\le t_0\wedge\bar{T}\delta_1^2t_0.
    \]

    To prove the claim, we need the followings to hold:
    \[
        \frac{c_1\gamma^2a^{-1}}{16^2}\ge1,
        \quad\frac{c_2\gamma^\frac{15}{4}a^{-2}}{16^4\delta_1^2}\ge1
        \quad\text{and}\quad\bar{T}\delta_1^2\ge1.
    \]
    This can be achieved by choosing $\delta_1=\max\{1,\bar{T}^{-\frac{1}{2}}\}$
            and
    \[
        \gamma=\max\{16(a c_1^{-1})^\frac{1}{2},32(a^2c_2^{-1}\delta_1^2)
                ^\frac{4}{15},4\delta_1+1\},
    \]
    which finishes the proof of the claim \ref{claim5.1}.
\end{proof}

    Let $U:=B_{g_0}(p,r_0-L\sqrt{t_0})$.
    By applying Proposition \ref{fact local RF} on $(U,g(t_0))$ with
            $\rho:=\sqrt{C_0^{-1}t_0}$,
            there exist constants $\alpha\in(0,1]$ and $\Lambda>1$ such
            that we can extend the Ricci flow $g(t)$ to $[0,(1+\nu)^2t_0]$ on
            $U_{g(t_0)}^\rho:=\{x\in U\mid B_{g(t_0)}(x,\rho)\subset\subset U\}$
            such that
    \[
        |\Rm|_{g(t)}\le\Lambda\rho^{-2} \text{ for }t\in[t_0,(1+\nu)^2t_0].
    \]
    where $(1+\nu)^2t_0:=t_0+\alpha\rho^2=(1+\alpha C_0^{-1})t_0$ with some
            $\nu(\mu,v_0)>0$.

    Now we choose $a:=\Lambda C_0(1+\nu)^2$.
    Then for all $t\in(0,(1+\nu)^2t_0]$, we have
    \[
        |\Rm|_{g(t)}\le\frac{\Lambda\rho^{-2}(1+\nu)^2t_0}{t}=
                \frac{a}{t}\text{ on }U_{g(t_0)}^\rho.
    \]

\begin{claim} \label{claim5.2}
    $U_{g(t_0)}^\rho\supset B_{g_0}(p,r_0-2L\sqrt{t_0})$.
\end{claim}

\begin{proof}[Proof of the Claim \ref{claim5.2}]
    For any $x\in B_{g_0}(p,r_0-2L\sqrt{t_0})$, we need to show that
            $B_{g(t_0)}(x,\rho)\subset\subset U$.
    To this end, we choose $L>\sqrt{C_0^{-1}}+\beta\sqrt{C_0}$.
    Then $\rho=\sqrt{C_0^{-1}t_0}<L\sqrt{t_0}-\beta\sqrt{C_0t_0}$.
    By the shrinking balls lemma \cite[Corollary 3.3]{ST22JDG},
    \begin{align*}
        B_{g(t_0)}(x,\rho)
                &\subset\subset B_{g(t_0)}(x,L\sqrt{t_0}-\beta\sqrt{C_0t_0}) \\
                &\subset B_{g_0}(x,L\sqrt{t_0})
                \subset B_{g_0}(p,r_0-L\sqrt{t_0})=U.
    \end{align*}
    This finishes the proof of the claim \ref{claim5.2}.
\end{proof}

    Now we have obtained a local Ricci flow with
    \[
        |\Rm|_{g(t)}\le\frac{a}{t}\text{ on }B_{g_0}(p,r_0-2L\sqrt{t_0})
                \times(0,(1+\nu)^2t_0].
    \]

    Suppose we have constructed a local Ricci flow with
    \[
        |\Rm|_{g(t)}\le\frac{a}{t}\text{ on }B_{g_0}(p,r_k)\times(0,t_k],
    \]
    where $t_k:=(1+\nu)^{2k}t_0$ and $r_k:=r_{k-1}-2L\sqrt{t_{k-1}}$ for
            $k=1,2,\cdots$.
    We can repeat the above argument to extend the Ricci flow to
            $B_{g_0}(p,r_{k+1})\times(0,t_{k+1}]$.
    This process is straightforward and we sketch some key points
            for readers' convenience:
    \begin{enumerate}
        \item[$1^\circ$] For any $x\in B_{g_0}(p,r_k-L\sqrt{t_k})$ and
                $t\in[0,t_k]$, we have
            \[
                B_{g(t)}(x,\gamma\sqrt{t_k})\subset B_{g_0}(p,r_k);
            \]
        \item[$2^\circ$] Apply Theorem \ref{thm t power k} on
                $B_{g(t)}(x,(\delta_1+\frac{\gamma}{4})\sqrt{t_k})$
                with proper scaling to obtain that
            \[
                \Rm_{g(t)}+\frac{t}{\delta_1^4t_k^2}\mathcal{I}_{g(t)}\in\mathfrak{C}_{\eta,\mu}
                        \text{ on }B_{g(t)}(x,\delta_1\sqrt{t_k})\times[0,t_k];
            \]
        \item[$3^\circ$] Whenever $\delta_1\sqrt{t_k}\le1$, we can apply Lemma
                \ref{lma pseudo.} to arrive that on $B_{g_0}(p,r_k-L\sqrt{t_k})$,
            \[
                |\Rm|_{g(t)}\le\frac{C_0}{t}
                        \text{ and }\Inj_{g(t)}(x)\ge\sqrt{C_0^{-1}t}
                        \text{ for all }t\le t_k;
            \]
        \item[$4^\circ$] Let $U_k:=B_{g_0}(p,r_k-L\sqrt{t_k})$.
            By the shrinking balls lemma \cite[Corollary 3.3]{ST22JDG},
            \[
                U_{k,g(t_k)}^\rho\supset B_{g_0}(p,r_k-2L\sqrt{t_k})
                        =B_{g_0}(p,r_{k+1});
            \]
        \item[$5^\circ$] By applying Proposition \ref{fact local RF} on
                $(U_k,g(t_k))$ with $\rho:=\sqrt{C_0^{-1}t_k}$, we can extend
                the Ricci flow to obtain that
            \[
                |\Rm|_{g(t)}\le\frac{a}{t}\text{ on }B_{g_0}(p,r_{k+1})
                        \times[0,t_{k+1}].
            \]
    \end{enumerate}

    The process stops at the $k$-th step where $r_{k+1}\le r_0-1=s_0-2$ or
            $\delta_1\sqrt{t_k}>1$.
    For the first case, we can choose $i\le k$ such that $r_i>r_0-1$ and
            $r_{i+1}\le r_0-1$ but $\delta_1\sqrt{t_i}\le1$.
    Then we have
    \[
        1\le r_0-r_{i+1}=2L(\sqrt{t_0}+\sqrt{t_1}+\cdots+\sqrt{t_i})
                \le\frac{2L\sqrt{t_i}}{1-(1+\nu)^{-1}},
    \]
    which implies that
    \[
        t_i\ge\frac{\nu^2}{4L^2(1+\nu)^2}=:\sigma_1(\eta,\mu,v_0)>0.
    \]
    If it is the second case, then $\delta_1\sqrt{t_{k-1}}\le1$ and
            $r_k>r_0-1$ but $\delta_1\sqrt{t_{k}}>1$.
    Then we have $t_k>\delta_1^{-2}$.

    In any cases, we have shown that there exist $T(\eta,\mu,v_0),D_1(\mu,v_0)>0$
            such that there exists a Ricci flow $g(t),t\in[0,T]$ on
            $B_{g_0}(p,s_0-2)$ with $g(0)=g_0$ satisfying
    \[
        |\Rm|_{g(t)}\le\frac{D_1}{t}\text{ on }B_{g_0}(p,s_0-2)\times(0,T].
    \]
    This completes the proof.
\end{proof}

\dummyline{1}

Then by a limiting argument, we obtain the following global existence for Ricci flow.

\begin{corollary} \label{cor global exist.0}
    For any $v_0>0$, $\eta\in(0,1]$ and $\mu\ge\eta+1$, there exist
            $T(\eta,\mu,v_0)>0$ and $D_1(\mu,v_0)>0$ such that
            the following holds.
    Suppose $(M,g_0)$ is a $4$-dimensional complete noncompact Riemannian manifold satisfying
    \[\left\{\;\begin{aligned}
        &\Rm_{g_0}\in\mathfrak{C}_{\eta,\mu}; \\
        &\VolB_{g_0}(x,1)\ge v_0\text{ for all }x\in M.
    \end{aligned}\right.\]
    Then there exists a complete solution of Ricci flow $(M,g(t))$
            on $t\in[0,T]$ with $g(0)=g_0$ such that
    \[\left\{\;\begin{aligned}
        &\sup\limits_{M}|\Rm|_{g(t)}\le\frac{D_1}{t}\text{ for all }t\in(0,T]; \\
        &\Rm_{g(t)}\in\mathfrak{C}_{\eta,\mu}.
    \end{aligned}\right.\]
\end{corollary}

\dummyline{0.5}

\begin{proof}
    Fix any $p\in M$, by applying Theorem \ref{thm local exist.0} to
            $B_{g_0}(p,R+2)$ for $R\ge2$, we obtain a Ricci flow $g_R(t)$
            on $B_{g_0}(p,R)$ for $t\in[0,T]$ with $g_R(0)=g_0$ such that
    \[
        \sup\limits_{B_{g_0}(p,R)}|\Rm|_{g_R(t)}\le\frac{D_1}{t}\text{ for all }t\in(0,T].
    \]

    Using the argument in the proof of \cite[Theorem 1.1]{CHL24AnnPDE},
            which is based on Chen's local estimates \cite{Chen09JDG}
            (see also \cite{Simon08IMRN}) and
            the modified Shi's higher order estimates \cite[Theorem 14.16]
            {CGIKLN08Book}, we can extract convergent subsequence in locally
            smooth sense to obtain a smooth complete solution $g(t)$ to the
            Ricci flow with $g(0)=g_0$ on $M\times[0,T]$.
    The curvature estimate passes to the limit and $\Rm_{g(t)}\in
            \mathfrak{C}_{\eta,\mu}$ follows from Corollary \ref{cor preseve.}.
\end{proof}

\dummyline{1}

Applying Corollary \ref{cor global exist.0}, we prove the following topological gap theorem. 

\begin{theorem} \label{thm diff. R4}
    Let $(M,g_0)$ be a complete noncompact $4$-dimensional Riemannian manifold with
            $\Rm_{g_0}\in\mathfrak{C}_{\eta,\mu}$ for some $\eta\in[0,1]$ and
            $\mu\ge\eta+1$ with $\mu>1$.
    If $(M,g_0)$ has maximal volume growth, then $M$ is diffeomorphic to
            $\mathbb{R}^4$.
\end{theorem}

\dummyline{0.5}

\begin{proof}
    It suffices to show the case $\eta>0$ since $\mathfrak{C}_{0,\mu}\subset
            \mathfrak{C}_{\varepsilon,\mu}$ for any small $\varepsilon>0$.
      By assumption, there exists $v_0>0$ such that
    \[
        \VolB_{R^{-2}g_0}(x,1)\ge v_0\text{ for all }R>0\text{ and for all }x\in M.
    \]
    Clearly, $\mathfrak{C}_{\eta,\mu}$ is scaling invariant.
    By applying Corollary \ref{cor global exist.0} to $R^{-2}g_0$ for
            $R\to\infty$ and rescaling it back, we obtain a sequence of Ricci
            flow $g_R(t)$ on $[0,TR^2]$ with $g_R(0)=g_0$ and
    \[
        |\Rm|_{g_R(t)}\le\frac{D_1}{t}\text{ for }t\in(0,TR^2].
    \]
    Combining with volume lower bound, the injectivity radius estimate
            follows from Cheeger-Gromov-Taylor \cite{CGT82JDG}.
    Then we can use the argument in the proof of \cite[Theorem 1.1]{CHL24AnnPDE}
            again to obtain a complete long-time solution $g(t)$ to the Ricci
            flow with $g(0)=g_0$ and desired injectivity radius estimate.
    Moreover, $\Rm_{g(t)}\in\mathfrak{C}_{\eta,\mu}$ for all $t\ge0$ by
            Corollary \ref{cor preseve.}.
    Then the result follows from \cite[Theorem 1.1]{HP25BLMS}.
\end{proof}

\dummyline{1}

Inspired by the work \cite{CLP24arXiv} by Chan, Lee and Peachey, we observe that if a Ricci flow coming out of a metric cone has curvature operator in $\mathfrak{C}_{\eta,\mu}$ and maximal volume growth, then it must be an expanding Ricci soliton. More precisely, we obtain the following. 

\begin{theorem} \label{thm soliton}
    Suppose $(M,g(t))$ is a $4$-dimensional complete noncompact Ricci flow on $M\times(0,+\infty)$
            such that for some $v_0,\alpha>0$, $\eta\in[0,1]$ and $\mu\ge\eta+1$
            with $\mu>1$, the following properties hold for all $t>0$:
    \begin{enumerate}
        \item[\textup{(a)}] $\Rm_{g(t)}\in\mathfrak{C}_{\eta,\mu}$;
        \item[\textup{(b)}] $|\Rm|_{g(t)}\le\alpha t^{-1}$;
        \item[\textup{(c)}] $\AVR_{g(t)}\ge v_0>0$.
    \end{enumerate}
    Suppose further that $(M,d_0,x_0)$ is isometric to a metric cone
            $(C(X),d_c,0)$, where $d_0$ is the well-defined metric on $M$ given
            by pointwise limit of $d_{g(t)}$ as $t\rightarrow0$
            (see \cite[Proposition 2.2]{CLP24arXiv}).
    Then there exists a smooth function $u$ on $M\times(0,+\infty)$ such that
    \begin{enumerate}
        \item[\textup{(\romannumeral1)}] $2t\Ric_{g(t)}+g(t)-2\nabla^2u=0$ on
                $M\times(0,+\infty)$;
        \item[\textup{(\romannumeral2)}] $u(\cdot,t)\to\frac{1}{4}
                d_0(x_0,\cdot)^2$ in $C_\mathrm{loc}^\beta(M)$ for some
                $\beta\in(0,1)$ as $t\rightarrow0$.
    \end{enumerate}
\end{theorem}


\dummyline{0.5}

\begin{proof}

In fact, the proof in \cite[Theorem 6.1]{CLP24arXiv} has already provided a more general result, although the statement of \cite[Theorem 6.1]{CLP24arXiv} assumes
            $\Rm_{g(t)}\in\mathfrak{C}_{\mathrm{WPIC1}}$. One can replace assumption \textup{(a)} above by that $\Rm_{g(t)}$ satisfies the $2$-nonnegative flag curvature condition for all $t>0$, which is ensured
            by Lemma \ref{lma 2-non flag} for our case. We sketch their proof here for readers' convenience.

The 2-nonnegative flag curvature, which implies nonnegative Ricci curvature,
        together with the curvature decay assumption \textup{(b)}, allows us
        to apply \cite[Proposition 2.2]{CLP24arXiv} to obtain the distance
        distortion estimate and the well-defined limit metric $d_0$ as $t\to0$.
Assumption \textup{(c)} provides the required non-collapsing, which is used
        for the Gaussian heat kernel estimates and the Cheeger-Jiang-Naber
        regularization \cite{CJN21AnnMath}.
Then one can apply \cite[Lemma 5.1]{CLP24arXiv} to obtain a sequence of
        square distance like functions $u_{i,0}$, which are then evolved
        by the heat equation to construct a sequence of approximations $v_i$.
By the uniform bounds and the higher order estimates of $v_i$ given by
        \cite[Proposition 5.3 and Lemma 5.4]{CLP24arXiv}, which are based on
        heat kernel estimates and Shi's estimates, one can extract a subsequence
        converging to the desired function $u$ as $t\to0$.

It remains to show that $u$ satisfies the expanding soliton equation.
Replace $g_i(t)$ by $g(t_i+t)$ and $v_i(t)$ by $v_i(t_i+t)$ and let
\[
    S_i(t):=-2t\Ric_{g_i(t)}-g_i(t)+2\nabla^{2,g_i(t)}v_i(t)\text{ for all }t>0,
\]
and $\varphi_i$ be the negative part of the lowest eigenvalue of
        $S_i(t)\owedge g_i(t)$ or similarly the positive part of the
        largest eigenvalue.
Since $v_i$ satisfies the heat equation and $\Rm_{g_i(t)}$ has 2-nonnegative
        flag curvature,  by the proof in
        \cite[Lemma 3.2]{CLP24arXiv}, one can obtain that
\[
    \Boxg{g_i(t)}\varphi_i\le\mathcal{R}_i\varphi_i,
\]
for $t>0$ whenever $\varphi_i\ge0$ in the barrier sense. We would like to point out that in the proof of \cite[Lemma 3.2]{CLP24arXiv}, the weakly PIC1 condition is only used to guarantee that $R_{1i1i}+R_{2i2i}\geq 0$ for all $i\geq 3$, which can be implied by  the 2-nonnegative
        flag curvature condition. 
With the above inequality established, one can utilize the local maximum
        principle \cite{LT22CJM} and the heat kernel estimates \cite{BCW19Inv}
        to show that $\varphi_i\to0$ as $i\to\infty$, which implies that
        the limit function $u$ satisfies the expanding soliton equation.
This completes the proof.
\end{proof}

\dummyline{1}

The following lemma tells us when $0\le\eta<\frac{9}{16}$, $\mathfrak{C}_{\eta,\mu}$ implies Ricci pinched. 

\begin{lemma} \label{lma Ricci pinched}
    Let $(M,g)$ be a Riemannian $4$-manifold with $\Rm\in\mathfrak{C}
            _{\eta,\mu}$ for some $\eta\in[0,\frac{9}{16})$ and $\mu\ge\eta+1$
            with $\mu>1$.
    Then $(M,g)$ is Ricci pinched in the following sense:
    \[
        \Ric\ge(1-\frac{4}{3}\sqrt{\eta})\frac{\mathcal{R}}{4}g\ge0.
    \]
\end{lemma}

\begin{proof}
    Choosing a local orientation and recalling \eqref{eqn norm Rc0}, we have
    \begin{align*}
        |\mathring{\Rc}|^2&=4|B|^2=4(B_1^2+B_2^2+B_3^2) \\
        &\le4(B_2+B_3)^2\le4\eta(A_1+A_2)(C_1+C_2) \\
        &\le\frac{16}{9}\eta\tr(A)\tr(C)=\frac{\eta}{9}\mathcal{R}^2.
    \end{align*}
    It follows that
    \[
        \Ric=\mathring{\Rc}+\frac{\mathcal{R}}{4}g
                \ge(1-\frac{4}{3}\sqrt{\eta})\frac{\mathcal{R}}{4}g\ge0.
    \]
    This completes the proof.
\end{proof}

\begin{remark}
    What we really need in the proof of Lemma \ref{lma Ricci pinched} are the
            first condition in the definition of $\mathfrak{C}_{\eta,\mu}$
            and an additional non-negativity assumption.
    More precisely, if $B=0$, it is enough to assume $\mathcal{R}\ge0$,
            otherwise we need $A_1+A_2\ge0$ or $C_1+C_2\ge0$, which is called
            half nonnegative isotropic curvature in \cite{CX23MathZ, CX25JMPA}.
    Under the first condition and the assumption $B\neq0$, these two
            alternatives are in fact equivalent, so half nonnegative isotropic curvature is actually nonnegative isotropic curvature in this case.
\end{remark}

\dummyline{1}

Combining this with Theorem \ref{thm soliton}, we obtain the following geometric gap theorem.

\begin{corollary} \label{cor flat}
    Suppose $(M,g_0)$ is a $4$-dimensional  complete noncompact Riemannian manifold with maximal volume growth and $\Rm_{g_0}\in\mathfrak{C}_{\eta,\mu}$ for some
            $\eta\in[0,\frac{9}{16})$ and $\mu\ge\eta+1$ with $\mu>1$.
    Then $(M,g_0)$ is isometric to flat Euclidean space.
\end{corollary}

\begin{proof}
    The proof follows \cite[Corollary 6.3]{CLP24arXiv} and we sketch main steps here. By Lemma \ref{lma 2-non flag}, we have $\Ric_{g_0}\ge0$.
    It follows from \cite[Theorem 5.3]{Gromov99Book} and
            \cite[Theorem 7.6]{CC96AnnMath} that $(M,R_i^{-2}g_0,x_0)$ converges
            to some metric cone $C(X)$ in the pointed Gromov-Hausdorff sense
            for some $x_0\in M$ as $R_i\to+\infty$.
    Since the assumption is scaling invariant, by the proof of Theorem
            \ref{thm diff. R4}, we can construct a sequence of complete
            long-time Ricci flow $g_i(t)$ with $g_i(0)=R_i^{-2}g_0$
            and $\Rm_{g_i(t)}\in\mathfrak{C}_{\eta,\mu}$,
            which implies Ricci pinched by Lemma \ref{lma Ricci pinched}.
    By Hamilton's compactness theorem \cite{Hamilton95AJM} and
            \cite[Proposition 2.2]{CLP24arXiv}, we obtain a smooth manifold
            $M_\infty$ and a long-time Ricci flow $g_\infty(t)$ on $M_\infty$
            coming out of $C(X)$.
    Moreover, $\AVR_{g_\infty(t)}=\AVR_{g_i(t)}=\AVR_{g_0}$ for all $t>0$,
            see \cite[Theorem 0.1]{Colding97AnnMath} and the proof of
            \cite[Theorem 7]{Yokota08GeomDedi} and \cite[Theorem 1.2]{SS13MathZ}.
    By Theorem \ref{thm soliton}, $(M_\infty,g_\infty(1),\nabla u(1))$ is an
            expanding gradient Ricci soliton which is also Ricci pinched,
            and therefore the scalar curvature decays exponentially at spatial
            infinity, see the proof of \cite[Proposition 1.8]{Deruelle17ASNSP}.
    By  virtue of the properties of $2$-nonnegative flag curvature, we also
            have $|\Rm|\le C_n\mathcal{R}$ as in \cite[Lemma A.2]{LT25GT}.
    By \cite[Theorem 2]{GPZ94IMRN} and \cite[Proposition 2.4]{Deruelle17ASNSP}, see also \cite[Theorem 1.3 and Theorem 1.4]{CL25AdvMath},
            we conclude that $\AVR_{g_\infty(1)}=1$ and hence $\AVR_{g_0}=1$. The proof is completed by Bishop–Gromov volume rigidity theorem.
   \end{proof}

\dummyline{1}

\dummynewpage
\section{Existence of Ricci flow II and applications}

In this section, we study volume non-collapsed manifolds with  a lower bound with respect to $\mathfrak{C}_{\eta,\mu}$. We will first show the following key differential inequality for the lower bound $l$ with respect to $\mathfrak{C}_{\eta,\mu}$. 

\begin{lemma} \label{lma evol. l}
    For $\mu-1\ge\eta>0$ and any smooth $4$-dimensional Ricci flow $(M,g(t)),t\in[0,T)$
            which is possibly incomplete, let
    \begin{equation} \label{eqn def. l}
        l(p,t):=\inf\{\alpha\in[0,\infty)\mid\Rm(p,t)
                +\alpha\mathcal{I}_{g(t)}\in\mathfrak{C}_{\eta,\mu}\}.
    \end{equation}
    Then there exists a constant $C_\eta>0$ depending only on $\eta$ such that for any $(p,t)\in M\times(0,T)$, we have
    \begin{equation} \label{eqn l bound}
        l(p,t)\le C_\eta|\Rm(p,t)|,
    \end{equation}
    and
    \begin{equation} \label{eqn box l}
        \Boxg{g(t)}l\le\mathcal{R}_{g(t)}l+6l^2,
    \end{equation}
    in the following barrier sense:
    for any $(p,t)\in M\times(0,T)$, we can find a neighborhood $\mathcal{U}
            \subset M\times(0,T)$ of $(p,t)$ and a smooth (lower barrier)
            function $\varphi:\mathcal{U}\to\mathbb{R}$ such that $\varphi\le l$
            in $\mathcal{U}$, with equality at $(p,t)$, and
    \[
        \Boxg{g(t)}\varphi\le\mathcal{R}_{g(t)}l+6l^2\quad\text{at }(p,t).
    \]
\end{lemma}

\dummyline{0.5}

\begin{proof}
    For any $(p,t)\in M\times(0,T)$, let
    \[
        \Lambda(p,t):=\Rm_{g(t)}(p)+l(p,t)\mathcal{I}_{g(t)}
                \in\mathfrak{C}_{\eta,\mu}.
    \]
    We may assume $\Lambda(p,t)\in\partial\mathfrak{C}_{\eta,\mu}$, otherwise
            we have $\Lambda(p,t)\in\interior(\mathfrak{C}_{\eta,\mu})$,
            then $l\equiv0$ in a neighborhood of $(p,t)$ and the results hold
            trivially.
    Assuming $\Lambda(p,t)\in\partial\mathfrak{C}_{\eta,\mu}$, we can obtain
            \eqref{eqn l bound} by the same argument as in the proof of
            \eqref{eqn varphi over Phi} with $\Phi=1$.
    It remains to show \eqref{eqn box l}.
    After choosing a local orientation, we have one of the following conditions
            holds at $(p,t)$:
    \begin{enumerate}
        \item[\textup{$1^\circ$}] $\hat{F}^1(p,t)=0,
                \quad\hat{F}^2(p,t)\ge0,\quad\hat{F}^3(p,t)\ge0$;
        \item[\textup{$2^\circ$}] $\hat{F}^1(p,t)\ge0,
                \quad\hat{F}^2(p,t)=0,\quad\hat{F}^3(p,t)\ge0$;
        \item[\textup{$3^\circ$}] $\hat{F}^1(p,t)\ge0,
                \quad\hat{F}^2(p,t)\ge0,\quad\hat{F}^3(p,t)=0$.
    \end{enumerate}

    \dummyline{0.5}

    For the first case, there exists $\xi\in\Theta_{p,t}(M)$ such that
    \[
        \eta X^{\Lambda}(\xi)Y^{\Lambda}(\xi)-Z^{\Lambda}(\xi)^2=0.
    \]

    Extend $\xi$ to a local section on a neighborhood
            $\mathcal{U}$ of $(p,t)$ and define the associated \emph{functions}.
    Denote
    \[\left\{\;\begin{aligned}
        X&:=X_{\xi}^{\Lambda}=X_{\xi}^{\Rm}+2l=:\x+2l; \\
        Y&:=Y_{\xi}^{\Lambda}=Y_{\xi}^{\Rm}+2l=:\y+2l; \\
        Z&:=Z_{\xi}^{\Lambda}=Z_{\xi}^{\Rm}=:\z.
    \end{aligned}\right.\]
    Then it holds at $(p,t)$ that
    \[
        0=XY-Z^2/\eta=4l^2+2(\x+\y)l+\x\y-\z^2/\eta.
    \]

    With the same arguments in the proof of Theorem \ref{thm t power k},
            we may assume $X,Y,Z>0$ at $(p,t)$.
    Then $l$ must be the bigger root of the above quadratic equation,
            otherwise it holds $l\le-\frac{\x+\y}{4}$, which contradicts
            $X+Y>0$.
    For any $(q,\tau)\in\mathcal{U}$, let
    \[
        \varphi(q,\tau):=\frac{-(\x+\y)+\sqrt{(\x-\y)^2+4\z^2/\eta}}{4}
                \text{ and }f:=(\x-\y)^2+4\z^2/\eta.
    \]
    Clearly, we have $\varphi=l$ and $\sqrt{f}=X+Y>0$ at $(p,t)$.
    We claim $\varphi$ is a lower barrier for $l$ in $\mathcal{U}$
            by shrinking $\mathcal{U}$ if necessary.
    Indeed, for any $(q,\tau)\in\mathcal{U}$, we have at $(q,\tau)$,
    \begin{align*}
        (\x+2l)(\y+2l)-\z^2/\eta\ge0=(\x+2\varphi)(\y+2\varphi)-\z^2/\eta,
    \end{align*}
    which implies that $l\ge\varphi$.

    By using \eqref{eqn box comp.1} with $\Phi=1$ and \eqref{eqn evol. Rm},
            we obtain at $(p,t)$,
    \begin{equation} \label{eqn z box z}
        \z\Box\z/\eta\le\frac{1}{2}(Y\Box\x+X\Box\y)+(X+Y)(\mathcal{R}+6l)l.
    \end{equation}

    Now we compute at $(p,t)$,
    \begin{align*}
        \Boxg{g(t)}\varphi&=-\frac{\Box\x+\Box\y}{4}+\frac{\Box f}{8\sqrt{f}}
                +\frac{|\nabla f|^2}{16f\sqrt{f}} \\
        &=-\frac{\Box\x+\Box\y}{4}+\frac{(\x-\y)(\Box\x-\Box\y)}{4\sqrt{f}}
                +\frac{4\z\Box\z/\eta}{4\sqrt{f}} \\
                &\qquad-\frac{|\nabla(\x-\y)|^2+4|\nabla\z|^2/\eta}{4\sqrt{f}}
                \cdot\frac{(\x-\y)^2+4\z^2/\eta}{f} \\
                &\qquad+\frac{|(\x-\y)\nabla(\x-\y)+4\z\nabla\z/\eta|^2}
                {4f\sqrt{f}} \\
        &\le-\frac{\Box\x+\Box\y}{4}+\frac{(X-Y)(\Box\x-\Box\y)}{4(X+Y)} \\
                &\qquad+\frac{Y\Box\x+X\Box\y}{2(X+Y)}+(R+6l)l \\
        &=Rl+6l^2,
    \end{align*}
    where we have used $\sqrt{f}=X+Y$ at $(p,t)$, \eqref{eqn z box z} and
            Cauchy-Schwarz inequality in the inequality step.

    The cases $2^\circ$ and $3^\circ$ are  simpler and we only sketch
            the proof of the case $2^\circ$ here.
    Suppose $\xi\in\Theta_{p,t}(M)$ satisfies $\mu X^\Lambda(\xi)
            -W^\Lambda(\xi)=0$ and extend it to a local section on a
            neighborhood $\mathcal{U}$ of $(p,t)$.
    Denote
    \[\left\{\;\begin{aligned}
        X&:=X_{\xi}^{\Lambda}=X_{\xi}^{\Rm}+2l=:\x+2l; \\
        W&:=W_{\xi}^{\Lambda}=W_{\xi}^{\Rm}+2l=:\w+2l.
    \end{aligned}\right.\]

    For any $(q,\tau)\in\mathcal{U}$, let
    \[
        \varphi(q,\tau):=-\frac{(\mu\x-\w)}{2(\mu-1)}.
    \]
    Clearly, $\varphi$ is a lower barrier for $l$ in $\mathcal{U}$ with
            $\varphi=l$ at $(p,t)$. By \eqref{eqn box comp.2} and
            \eqref{eqn evol. Rm}, we have at $(p,t)$,
    \[
        \Boxg{g(t)}\varphi=-\frac{\mu\Box\x-\Box\w}{2(\mu-1)}
                \le\mathcal{R}l+6l^2.
    \]
    This completes the proof.
\end{proof}

\dummyline{1}

With Lemma \ref{lma evol. l}, following the argument in \cite{BCW19Inv}, we obtain the following pseudo-locality theorem for $\mathfrak{C}_{\eta,\mu}$. 

\begin{theorem} \label{thm bcw thm1}
    For any $v_0>0$, $\alpha_0\in[0,1]$, $\eta\in(0,1]$ and $\mu\ge\eta+1$,
            there exist $\hat{T}(\mu,v_0)>0$ and $D_1(\mu,v_0)>0$ such that the
            following holds.
    Let $(M,g(t)),t\in[0,T)$ be a smooth $4$-dimensional complete Ricci flow
            with \emph{bounded curvature} satisfying
    \[
        \VolB_{g(0)}(p,1)\ge v_0\text{ for all }p\in M\text{ and }
                \Rm_{g(0)}+\alpha_0\mathcal{I}_{g(0)}\in\mathfrak{C}_{\eta,\mu}.
    \]
    Then for all $t\in(0,T\wedge\hat{T}]$, we have
    \[
        \Rm_{g(t)}+D_1\alpha_0\mathcal{I}_{g(t)}\in\mathfrak{C}_{\eta,\mu}
                \text{ and }\sup\limits_{M}|\Rm|_{g(t)}\le\frac{D_1}{t}.
    \]
\end{theorem}


\dummyline{0.5}

\begin{proof}
    With \eqref{eqn box l} in Lemma \ref{lma evol. l} established, the proof
            follows that of \cite[Theorem 1]{BCW19Inv} except that
            we use Lemma \ref{lma ancient sol.} instead of
            \cite[Proposition 11.4]{Perelman02arXiv} to obtain the contradiction.
    We sketch the proof here for readers' convenience.

    Let $l(p,t)$ be defined as in \eqref{eqn def. l} and assume, by rescaling
            if necessary, that $l(\cdot,0)\le\alpha_0\le\varepsilon_0$
            for some $\varepsilon_0(\mu,v_0)>0$ to be chosen later.
    Consider the maximal time interval $[0,t_1)$  on which $l\le1$ and a uniform
            unit ball volume lower bound holds.
    The point-picking technique and blow-up argument give the curvature estimate
            $|\Rm|_{g(t)}\le D_1/t$ for some $D_1(\mu,v_0)>0$.
    Indeed, if this failed, a sequence of parabolic rescalings $\tilde{g}_i(t)$
            would converge, after passing to a subsequence, to a non-flat
            ancient solution $\tilde{g}_\infty(t)$ with bounded curvature
            and Euclidean volume growth.
    Moreover, $\Rm_{\tilde{g}_\infty(t)}\in\mathfrak{C}_{\eta_\infty,\mu}$
            for some $\eta_\infty\in[0,\min\{1,\mu-1\}]$, which contradicts
            Lemma \ref{lma ancient sol.}.
    Here we use similar arguments in the proof of Lemma \ref{lma pseudo.}
            to show that $D_1$ does not depend on $\eta$.

    With the curvature estimate, as well as the volume lower bound, the
            heat kernel argument from \cite{BCW19Inv} can be applied to show that
            $l\le D_2\alpha_0\leq D_2\varepsilon_0$ for some $D_2(\mu,v_0)>0$.
    A suitable choice of $\varepsilon_0$ ensures that $l\le1/2$ for all $t<t_1$.
    By \cite[Corollary 6.2]{Simon12}, the volume lower bound is almost preserved
            for some uniform time, which implies that $t_1$ has a uniform lower
            bound $\hat{T}(\mu,v_0)>0$. Finally, let $D_1$ be the maximum of $D_1$ and $D_2$ above. 
    This completes the proof.
\end{proof}

\dummyline{1}

\begin{corollary} \label{cor BCW1}
    Let $\eta,\mu$ be as in Theorem \ref{thm bcw thm1}.
    For any $v_0,D>0$, there exists $\varepsilon(\eta,\mu,v_0,D)>0$
            such that the following holds.
    Let $(M,g)$ be a compact $4$-dimensional Riemannian manifold  with
    \[
        \mathrm{diam}_g(M)\le D,\text{ }\Vol_g(M)\ge v_0
                \text{ and }\Rm_g+\varepsilon I_g\in\mathfrak{C}_{\eta,\mu}.
    \]
   Then it also admits a metric whose curvature operator lies
            in $\mathfrak{C}_{\eta,\mu}$.
\end{corollary}


\dummyline{0.5}

\begin{proof}
    The proof follows that of \cite[Corollary 3]{BCW19Inv}.
    If the conclusion fails, we can obtain a sequence of compact
            $4$-dimensional Riemannian manifolds $\{(M_i,g_i)\}_{i=1}^\infty$,
            satisfying the assumptions with $\varepsilon=1/i$, but $M_i$ admits no metric with curvature operator in $\mathfrak{C}_{\eta,\mu}$. 
    By Shi's existence theorem \cite{Shi89JDG} and Theorem \ref{thm bcw thm1},
            we can construct a sequence of Ricci flows which would converge,
            after passing to a subsequence, to a Ricci flow
            $(M_\infty,g_\infty(t))_{t\in(0,\hat{T}]}$ with
            $\Rm_{g_\infty(t)}\in\mathfrak{C}_{\eta,\mu}$,
            where $\hat{T}(\mu,v_0, D)>0$ is given by Theorem \ref{thm bcw thm1}, noting that here we haved used  Bishop–Gromov volume comparison to guarantee the uniform lower bound $\tilde{v_0}(v_0, D)$ of volume of unit balls.
    Moreover, the uniform diameter bound ensures that $M_\infty$ is
            diffeomorphic to $M_i$ provided $i$ is large enough,
            which contradicts the choice of $M_i$.
\end{proof}

\dummyline{1}

\begin{corollary} \label{cor BCW2}
    Let $\eta,\mu$ be as in Theorem \ref{thm bcw thm1}.
    Suppose $v_0>0$ and $(X,d_X)$ is the Gromov-Hausdorff limit of a sequence
            $\{(M_i,g_i)\}_{i=1}^\infty$ of compact $4$-dimensional Riemannian manifolds
            satisfying
    \[
        \Vol_{g_i}(M_i)\ge v_0,\text{ }
                \Rm_{g_i}+\varepsilon_i I_{g_i}\in\mathfrak{C}_{\eta,\mu}
                \text{ and }\mathrm{diam}_{g_i}(M_i)\le D,
    \]
    for some sequence $\{\varepsilon_i\}\subset(0,1]$ with
            $\varepsilon_i\to\varepsilon_\infty$ as $i\to\infty$.
    Then there exist $\tau(\mu,v_0, D)>0$, a smooth manifold $M_\infty$ and a smooth
            Ricci flow $(M_\infty,g_\infty(t))_{t\in(0,\tau)}$ which satisfies
    \[
        \Rm_{g_\infty(t)}+D_1\varepsilon_\infty I_{g_\infty(t)}
                \in\mathfrak{C}_{\eta,\mu},
    \]
    and is coming out of the (possibly singular) space $(X,d_X)$ in the sense
            that
    \begin{equation} \label{eqn lim X dx}
        \lim_{t\rightarrow0}d_\mathrm{GH}\bigl((X,d_X),(M_\infty,d_{g_\infty(t)})
                \bigr)=0.
    \end{equation}
    Here $D_1$ is the constant given by Theorem \ref{thm bcw thm1}.
    In particular, for $\varepsilon_\infty=0$, we have $\Rm_{g_\infty(t)}
            \in\mathfrak{C}_{\eta,\mu}$ for all $t\in(0,\tau)$.
    Moreover, for any choice of $\varepsilon_\infty$, the space $X$ is
            homeomorphic to the manifold $M_\infty$ and $d_{g_\infty(t)}$
            converges uniformly to a distance function $d_0$ on $M_\infty$
            as $t\rightarrow0$ such that $(M_\infty,d_0)$ is isometric to
            $(X,d_X)$.
\end{corollary}


\dummyline{0.5}

\begin{proof}
    The proof follows that of \cite[Corollary 4]{BCW19Inv}.
    Arguing as in the proof of the previous corollary, we obtain a limiting
            Ricci flow $(M_\infty,g(t))_{t\in(0,\hat{T}]}$ with
            $\Rm_{g(t)}+D_1\varepsilon_\infty I_{g(t)}
            \in\mathfrak{C}_{\eta,\mu}$, where $\hat{T}(\mu,v_0,D),D_1(\mu,v_0,D)>0$
            are given by Theorem \ref{thm bcw thm1}.
    The proof of \eqref{eqn lim X dx} follows as in \cite[Theorem 7.2]
            {Simon09JRAM} by applying twice the triangle inequality combined
            with a two sided distance distortion control given by
            \cite[Lemma 6.1]{Simon09JRAM}.

    The Ricci curvature lower bound and the curvature decay estimate give a distance
            distortion of the form
    \[
        e^{-C(t_2-t_1)}d_{g(t_2)}(x,y)\le d_{g(t_1)}(x,y)\le
                d_{g(t_2)}(x,y)+C\sqrt{t_2},
    \]
    for any $x,y\in M_\infty$ and $0<t_1<t_2$, where $C(\mu,v_0,D)>0$.
    Then the limit $d_0(x,y):=\lim\limits_{t\to0}d_{g(t)}(x,y)$ exists and
    \begin{equation} \label{eqn dist. d0}
        e^{-Ct}d_{g(t)}(x,y)\le d_0(x,y)\le d_{g(t)}(x,y)+C\sqrt{t},
    \end{equation}
    for all $t>0$, which implies that $(M_\infty,d_0)$, as a metric space,
            is isometric to $(X,d_X)$ by \eqref{eqn lim X dx}.
    For some fixed $t>0$, the first inequality in \eqref{eqn dist. d0} implies
            that the identity map $(M_\infty,d_0)\to(M_\infty,d_{g(t)})$ is
            continuous, and its inverse is also continuous by the second
            inequality in \eqref{eqn dist. d0} combined with an approximation
            argument.
    This completes the proof.
\end{proof}

\dummyline{1}

Now let us come back to the case of noncompact manifolds. Since we may not have sectional curvature lower bound in our case, we can not construct Ricci flows by a simpler conformal change argument directly.  Thanks to the technique developed by Lai in \cite{Lai19AdvMath}, we are able to obtain the following short-time existence result for Ricci flow locally.

\begin{theorem} \label{thm local exist.1}
    For any $\alpha_0\in(0,1]$, $v_0>0$, $\eta\in(0,1]$ and $\mu\ge\eta+1$,
            there exist $T(\eta,\mu,\alpha_0,v_0)>0$, $D_1(\mu,v_0)>0$ and
            $D_2(\eta,\mu,v_0)>0$ such that the following holds.
    Let $(M,g_0)$ be a $4$-dimensional Riemannian manifold.
    Suppose $B_{g_0}(p,s_0)\subset\subset M$ for some $p\in M$ and $s_0\ge4$
            such that
    \[\left\{\;\begin{aligned}
        &\Rm_{g_0}+\alpha_0\mathcal{I}_{g_0}\in\mathfrak{C}_{\eta,\mu}
                \text{ on }B_{g_0}(p,s_0); \\
        &\VolB_{g_0}(x,1)\ge v_0\text{ for all }x\in B_{g_0}(p,s_0-1).
    \end{aligned}\right.\]
    Then there exists a Ricci flow $g(t)$ on $B_{g_0}(p,s_0-2)$
            for $t\in[0,T]$ with $g(0)=g_0$ such that
    \[\left\{\;\begin{aligned}
        &\sup\limits_{B_{g_0}(p,s_0-2)}|\Rm|_{g(t)}\le\frac{D_1}{t}\text{ for all }t\in(0,T]; \\
        &\Rm_{g(t)}+D_2\alpha_0\mathcal{I}_{g(t)}\in\mathfrak{C}_{\eta,\mu}.
    \end{aligned}\right.\]
\end{theorem}


\dummyline{0.5}

\begin{proof}
    The proof follows that of \cite[Theorem 1.1]{Lai19AdvMath}.
    Although \cite[Theorem 1.1]{Lai19AdvMath} assumes a lower bound with respect to
            $\mathfrak{C}_\mathrm{WPIC1}$, its proof only relies on the
            following three properties of $\mathfrak{C}_\mathrm{WPIC1}$:
             Ricci curvature  lower bound,  curvature decay
            lemma  and  evolution
            inequality of the lower bound $l$ with respect to the curvature cone, which are all satisfied by
            $\mathfrak{C}_{\eta,\mu}$ according to Lemma \ref{lma 2-non flag},
            Lemma \ref{lma pseudo.} and Lemma \ref{lma evol. l}.
    We sketch key steps of the proof here.

    Since we have the Ricci lower bound, by reducing $v_0$, we may assume that
            $\VolB_{g_0}(x,r)\ge v_0r^4$ for all $r\in(0,1]$.
    By rescaling, we may also assume that $\alpha_0\le\frac{1}{2D_2}<1$
            for some $D_2(\eta,\mu,v_0)>2$ to be chosen later.
    We want to use induction to construct a Ricci flow in expansion
            $(\{M_j\}_{j=1}^i,\{g_j(t)\}_{j=1}^i,\nu)$,
            see \cite[Definition 5.3]{Lai19AdvMath},
            satisfying the following conditions:
    \begin{enumerate}
        \item[\textup{(APA 1)}] Restricting it on $B_{g_0}(x_0,r_i)$, we get
                a smooth Ricci flow $g(t)$ up to $t_{i+1}$;
        \item[\textup{(APA 2)}] For each complete Ricci flow $(M_j,g_j(t))$,
                we have $|\Rm|_{g_j(t)}\le\frac{D_1}{t}$;
        \item[\textup{(APA 3)}] $l(x,t)\le D_2\alpha_0<1$ for all $(x,t)\in
                B_{g_0}(x_0,r_i)\times[0,t_{i+1}]$.
    \end{enumerate}
    Here $l(x,t)$ is defined as in \eqref{eqn def. l} and the constants
            $D_1(\mu,v_0),\nu(\mu,v_0)>0$ will be specified later.

    The initial step of the induction and the verification of the first two
            conditions are similar to the proof of Theorem
            \ref{thm local exist.0}.
    We first use conformal completion and Shi's existence theorem to construct
            a Ricci flow for some small time $t_1$.
    By shrinking $t_1$ if necessary, we can assume these three conditions hold
            for $i=0$.
    Provided (APA 3) holds, we can apply Lemma \ref{lma pseudo.} to
            a smaller ball to improve the curvature decay estimate.
    Then, using conformal completion, Shi's existence theorem and the
            doubling time estimate (or equivalently, invoking Proposition
            \ref{fact local RF}), we can extend the Ricci flow satisfying
            (APA 1) and (APA 2) for a longer time.
    The constants $D_1$ and $\nu$ are specified during the process of applying
            Lemma \ref{lma pseudo.} and doubling time estimate.

    The verification of (APA 3) is ingenious and introduces the generalized heat
            kernel, see \cite[Definition 5.4]{Lai19AdvMath}, with a similar
            Gaussian upper bound, see \cite[Proposition 5.5]{Lai19AdvMath}.
    Lemma \ref{fact lower vol. ctrl.} gives the preservation of the volume
            lower bound, which combined with the curvature decay estimate
            (APA 2) enables the invocation of the heat kernel estimates.
    With Lemma \ref{lma evol. l} established, we can use the heat kernel
            arguments to improve the upper bound of $l$ from \eqref{eqn l bound}
            to a rough uniform bound depending only on $\eta,\mu$ and $v_0$.
    The rough bound gives a better way to control the evolution of $l$,
            more precisely, from \cite[(7.3)]{Lai19AdvMath} to
            \cite[(7.11)]{Lai19AdvMath}, and a Ricci curvature lower bound, which will
            be used to obtain cut-off functions and volume controls
            in the next step.
    It remains to convert this rough bound to the stronger bound in (APA 3).
    For any fixed $x$ in some smaller ball and $t\in(t_{i+1},t_{i+2}]$, let
    \[
        \Phi(x,t,s):=\int_{B_{g_0}(x,3\sqrt[4]{t_{i+2}})}G(x,t;y,s)
                \mathcal{L}(y,s)\phi_{i+1}(y,s)d_sy,
    \]
    where $G$ denotes the generalized heat kernel, $\mathcal{L}(\cdot,t)
            :=e^{-Ct}l(\cdot,t)$ for some $C$ depending on the rough bound
            and $\phi_{i+1}$ is the cut-off function given by
            \cite[Lemma 4.1]{Lai19AdvMath}.
    Then $\mathcal{L}(x,t)=\lim\limits_{s\to t^-}\Phi(x,t,s)$ and hence
    \[
        \mathcal{L}(x,t)\le\Phi(x,t,t_1)
                +\int_{t_1}^t\frac{\partial^+}{\partial s}\Phi(x,t,s)ds.
    \]
    By the estimates of cut-off functions in \cite[Lemma 4.1]{Lai19AdvMath}
            and the gradient estimates of the heat kernel in
            \cite[Claim 5.8]{Lai19AdvMath}, as well as the evolution inequality
            of $\mathcal{L}$ and the heat kernel estimates, one eventually
            obtains that
    \[
        \mathcal{L}(x,t)\le C\exp\left(-\frac{1}{C\sqrt{t_{i+2}}}\right)
                +2\alpha_0C,
    \]
    for some $C(\eta,\mu,v_0)>0$.
    Since there exists $T(\eta,\mu,\alpha_0,v_0)>0$ such that the first term is
            bounded by $\alpha_0$ for $t_{i+2}\le T$, then by choosing
            $D_2(\eta,\mu,v_0)$ large enough, it holds that
            $l\le D_2\alpha_0$ as in (APA 3).
    Moreover, Lai's delicate choice of $r_{i}$ ensures that $r_0-r_{i+1}\le1$
            provided $T(\eta,\mu,\alpha_0,v_0)$ is small enough.
    This completes the induction step and hence the proof.
\end{proof}

\dummyline{1}

By taking a limit, we have:

\begin{corollary} \label{cor glabol exist.1}
       For any $\alpha_0\in(0,1]$, $v_0>0$, $\eta\in(0,1]$ and $\mu\ge\eta+1$,
            there exist $T(\eta,\mu,\alpha_0,v_0)>0$, $D_1(\mu,v_0)>0$ and
            $D_2(\eta,\mu,v_0)>0$ such that the following holds.
    Let $(M,g_0)$ be a $4$-dimensional complete noncompact Riemannian manifold
            (with possibly unbounded curvature) such that
    \[\left\{\;\begin{aligned}
        &\Rm_{g_0}+\alpha_0\mathcal{I}_{g_0}\in\mathfrak{C}_{\eta,\mu}; \\
        &\VolB_{g_0}(x,1)\ge v_0\text{ for all }x\in M.
    \end{aligned}\right.\]
    Then there exists a complete Ricci flow $g(t)$ on $M\times[0,T]$ with
            $g(0)=g_0$ such that
    \[\left\{\;\begin{aligned}
        &\sup\limits_{M}|\Rm|_{g(t)}\le\frac{D_1}{t}\text{ for all }t\in(0,T]; \\
        &\Rm_{g(t)}+D_2\alpha_0\mathcal{I}_{g(t)}\in\mathfrak{C}_{\eta,\mu}.
    \end{aligned}\right.\]
\end{corollary}

\dummyline{0.5}

\begin{proof}
    Given the local existence result in Theorem \ref{thm local exist.1},
            the proof of global version is  standard.
    See the proof of Corollary \ref{cor global exist.0} or
            \cite[Corollary 1.2]{Lai19AdvMath} for details.
\end{proof}

\dummyline{1}

As an application, we obtain the following regularity result for Gromov-Hausdorff limits of  complete noncompact volume non-collapsed manifolds with a lower bound with respect to $\mathfrak{C}_{\eta,\mu}$. 

\begin{corollary} \label{cor pGH Holder}
    Suppose $\alpha_0\in(0,1]$, $v_0>0$, $\eta\in(0,1]$ and $\mu\ge\eta+1$.
    Let $(M_i,g_i)$ be a sequence of $4$-dimensional complete noncompact Riemannian
            manifolds such that for all $i$,
    \[\left\{\;\begin{aligned}
        &\Rm_{g_i}+\alpha_0\mathcal{I}_{g_i}\in\mathfrak{C}_{\eta,\mu}; \\
        &\VolB_{g_i}(x,1)\ge v_0\text{ for all }x\in M_i.
    \end{aligned}\right.\]
    Then there exist a smooth manifold $M$, a point $x_\infty\in M$ and
            a continuous distance metric $d_0$ on $M$ such that for some
            points $x_i\in M_i$, a subsequence of $(M_i,d_{g_i},x_i)$ converges
            to $(M,d_0,x_\infty)$ in pointed Gromov-Hausdorff sense.
    Furthermore, the metric space $(M,d_0)$ is locally bi-H\"older homeomorphic to the
            smooth manifold $M$ equipped with any smooth metric.
\end{corollary}



\dummyline{0.5}

\begin{proof}
    The proof follows that of \cite[Corollary 1.3]{Lai19AdvMath},
            which is similar to that of Corollary \ref{cor BCW2}.
    We can apply Corollary \ref{cor glabol exist.1} and Lemma
            \ref{fact lower vol. ctrl.} to each $(M_i,g_i)$ to obtain a
            sequence of Ricci flows $\{(M_i,g_i(t),x_i)\}_{t\in(0,T]}$,
            which would converge, after passing to a subsequence,
            to a Ricci flow $(M,g(t),x_\infty)_{t\in(0,T]}$
            with the estimates for curvature, volume lower bound and distance
            distortion:
    \[\left\{\;\begin{aligned}
        &\Rm_{g(t)}+D\alpha_0\mathcal{I}_{g(t)}\in\mathfrak{C}_{\eta,\mu}; \\
        &\VolB_{g(t)}(x,1)\ge v\text{ for all }x\in M; \\
        &|\Rm|_{g(t)}\le\frac{D}{t},
    \end{aligned}\right.\]
    and
    \[
        d_{g(t_1)}(x,y)-D(\sqrt{t_2}-\sqrt{t_1})\le d_{g(t_2)}(x,y)
                \le e^{D(t_2-t_1)}d_{g(t_1)}(x,y),
    \]
    for any $x,y\in M$ and $0<t_1\le t_2\le T$, where $D(\eta,\mu,v_0)$,
            $T(\eta,\mu,\alpha_0,v_0)$ and $v(\eta,\mu,\alpha_0,v_0)$
            are positive constants.
    The distance distortion estimate implies that $d_{g(t)}$ converges locally
            uniformly to some metric $d_0$ as $t\to0$.
    Combined with that $d_{g_i(t)}$ converges locally uniformly to $d_{g_i}$,
            we have that $(M_i,d_{g_i},x_i)\to(M,d_0,x_\infty)$
            in the pointed Gromov-Hausdorff sense.
    The assertion of bi-H\"older homeomorphism follows from \cite[Lemma 2.3]
            {Lai19AdvMath}, see also the proof of \cite[Theorem 1.4]{ST21GT}.
    This completes the proof.
\end{proof}

\dummyline{1}


\appendix

\section{Algebraic curvature operators}
In this appendix, we review the notion of algebraic curvature operators and the evolution equation of the curvature operator under Ricci flow.
One may also read some textbooks or literature for more complete details, see for example \cite{CK04Book, Brendle10Book, AH2011, Wilking13JRAM, BCW19Inv}.

Consider the Lie algebra $\mathfrak{so}(n)$ of real-valued skew-symmetric matrices,
        the inner product and Lie bracket are defined by
\[
    \langle u,v\rangle:=-\frac{1}{2}\mathrm{tr}(uv)
            \quad\text{and}\quad[u,v]:=uv-vu
            \quad\text{for any }u,v\in\mathfrak{so}(n).
\]
Let $\{e_i\}_{i=1}^n$ be an orthonormal basis of $\mathbb{R}^n$.
We can identify $\wedge^2\mathbb{R}^n$ with $\mathfrak{so}(n)$ via
        the linear transformation determined by
\[
    e_i\wedge e_j\mapsto E_{ij}-E_{ji},
\]
where $E_{ij}$ is the matrix with $1$ at the $(i,j)$-entry and $0$ elsewhere.
The induced inner product on $\wedge^2\mathbb{R}^n$ is given by
\begin{align*}
    \langle e_i\wedge e_j,e_k\wedge e_l\rangle
            &=-\frac{1}{2}\tr((E_{ij}-E_{ji})(E_{kl}-E_{lk})) \\
    &=-\frac{1}{2}\tr(\delta_{jk}E_{il}-\delta_{jl}E_{ik}
            -\delta_{ik}E_{jl}+\delta_{il}E_{jk}) \\
    &=\delta_{ik}\delta_{jl}-\delta_{il}\delta_{jk}.
\end{align*}
Then the set $\{e_i\wedge e_j\}_{i<j}$ forms an orthonormal basis of
        $\wedge^2\mathbb{R}^n$.
The Lie bracket on $\wedge^2\mathbb{R}^n$ is given by
\begin{equation}\begin{aligned} \label{eqn Lie def.}
    \relax[e_i\wedge e_j,e_k\wedge e_l]&=(E_{ij}-E_{ji})(E_{kl}-E_{lk})
            -(E_{kl}-E_{lk})(E_{ij}-E_{ji}) \\
    &=\delta_{jk}E_{il}-\delta_{jl}E_{ik}-\delta_{ik}E_{jl}+\delta_{il}E_{jk} \\
            &\qquad-\delta_{il}E_{kj}+\delta_{jl}E_{ki}
            +\delta_{ik}E_{lj}-\delta_{jk}E_{li} \\
    &=\delta_{jk}(e_i\wedge e_l)-\delta_{jl}(e_i\wedge e_k)
            -\delta_{ik}(e_j\wedge e_l)+\delta_{il}(e_j\wedge e_k).
\end{aligned}\end{equation}
We identify $\wedge^2\mathbb{R}^n$ with its dual space
        $(\wedge^2\mathbb{R}^n)^*$ via the inner product and the corresponding
        isomorphism $u\mapsto\langle u,\cdot\rangle$.
Choose a basis $\{\varphi^\alpha\}$ of $(\wedge^2\mathbb{R}^n)^*$ and
        let $C_\gamma^{\alpha\beta}$ be the Lie structure constants defined by
\[
    [\varphi^\alpha,\varphi^\beta]=:C_\gamma^{\alpha\beta}\varphi^\gamma.
\]
Let $C_{(pq)}^{(ij)(kl)}$ be the structure constants with respect to the basis
        $\{e_i\wedge e_j\}_{i<j}$.
By \eqref{eqn Lie def.}, we have
\begin{equation} \label{eqn Lie str. const.}
    C_{(pq)}^{(ij)(kl)}=\delta^{jk}\mathcal{I}_{pq}^{il}
            -\delta^{jl}\mathcal{I}_{pq}^{ik}
            -\delta^{ik}\mathcal{I}_{pq}^{jl}
            +\delta^{il}\mathcal{I}_{pq}^{jk},
\end{equation}
where we write $\mathcal{I}_{ij}^{kl}:=\delta_i^k\delta_j^l
        -\delta_i^l\delta_j^k$ for notational convenience.
In this paper, we use common Latin letters to denote indices of
        $\mathbb{R}^n$ and Greek letters to denote indices of
        $\wedge^2\mathbb{R}^n$.

\dummyline{0.5}

Let $S_B^2(\mathfrak{so}(n))$ be the set of algebraic curvature
        operators on $\mathbb{R}^n$, i.e. the space of symmetric bilinear forms
        on $\mathfrak{so}(n)$ satisfying the first Bianchi identity and
every $\Rm\in S_B^2(\mathfrak{so}(n))$ can also be viewed as a
        self-adjoint linear operator on $\mathfrak{so}(n)$ or
        $\wedge^2\mathbb{R}^n$ by
\begin{equation}\begin{aligned} \label{eqn Rm def.}
    \langle\Rm(X\wedge Y),Z\wedge W\rangle
            &=\Rm(X\wedge Y,Z\wedge W) \\
    &=R(X,Y,Z,W)\quad\text{for any }X,Y,Z,W\in\mathbb{R}^n,
\end{aligned}\end{equation}
where $R(X,Y,Z,W)$ denotes the corresponding $(4,0)$ curvature tensor acting
        on $4$ vectors.

\dummyline{0.5}

The square of $\Rm$ is defined by
\[
    \Rm^2:=\Rm\circ\Rm: \wedge^2\mathbb{R}^n\to\wedge^2\mathbb{R}^n.
\]
Let $\{\varphi_\alpha\}$ be a basis of $\wedge^2\mathbb{R}^n$
        and $\{\varphi^\alpha\}$ the dual basis of $(\wedge^2\mathbb{R}^n)^*$.
Given a symmetric bilinear form $L$ on $\wedge^2\mathbb{R}^n$, we write
        $L_{\alpha\beta}:=L(\varphi_\alpha,\varphi_\beta)$.
Recalling that $C_\gamma^{\alpha\beta}$ are the Lie structure constants
        with respect to the basis $\{\varphi^\alpha\}$, the commutative bilinear
        operator $\#$ for curvature operators $M$ and $N$ is defined by
\begin{equation} \label{eqn sharp def.}
    (M\#N)_{\alpha\beta}:=\frac{1}{2}C_\alpha^{\gamma\eta}
            C_\beta^{\delta\theta}M_{\gamma\delta}N_{\eta\theta}.
\end{equation}
We write $M^\#:=M\#M$ for simplicity.

Now we  compute $\Rm^2$ and $\Rm^\#$ in terms of the
        components of $\Rm$.
Let $\{e_i\}$ be an orthonormal basis of $\mathbb{R}^n$ and
        $\{e_i\wedge e_j\}_{i<j}$ the corresponding orthonormal basis of
        $\wedge^2\mathbb{R}^n$.
By \eqref{eqn Rm def.}, we have \[R_{ijkl}:=\Rm(e_i,e_j,e_k,e_l)
        =\langle\Rm(e_i\wedge e_j),e_k\wedge e_l\rangle\] and hence
\[
    \Rm(e_i\wedge e_j)=\sum_{k<l}R_{ijkl}e_k\wedge e_l.
\]
Then for any $X=X^ie_i,Y=Y^je_j,Z=Z^ke_k,W=W^le_l$ in $\mathbb{R}^n$, we have
\begin{align*}
    &\Rm^2(X,Y,Z,W):=\langle\Rm^2(X\wedge Y),Z\wedge W\rangle \\
    &\qquad=X^iY^jZ^kW^l\sum_{p<q,\;r<s}R_{ijpq}R_{pqrs}
            \langle e_r\wedge e_s,e_k\wedge e_l\rangle \\
    &\qquad=\frac{1}{2}\sum_{p,q=1}^nR(X,Y,e_p,e_q)R(e_p,e_q,Z,W).
\end{align*}
For $\Rm^\#$, by \eqref{eqn Lie str. const.}, we have
\begin{align*}
    C_{(ij)}^{(xy)(pq)}R_{xyzw}&=(\delta^{yp}\mathcal{I}_{ij}^{xq}
            -\delta^{yq}\mathcal{I}_{ij}^{xp})R_{xyzw}
            -(\delta^{xp}\mathcal{I}_{ij}^{yq}
            -\delta^{xq}\mathcal{I}_{ij}^{yp})R_{xyzw} \\
    &=(\delta^{yp}\mathcal{I}_{ij}^{xq}
            -\delta^{yq}\mathcal{I}_{ij}^{xp})R_{xyzw}
            -(\delta^{yp}\mathcal{I}_{ij}^{xq}
            -\delta^{yq}\mathcal{I}_{ij}^{xp})R_{yxzw} \\
    &=2(\delta^{yp}\mathcal{I}_{ij}^{xq}
            -\delta^{yq}\mathcal{I}_{ij}^{xp})R_{xyzw},
\end{align*}
and then
\begin{align*}
    C_{(ij)}^{(xy)(pq)}R_{xyzw}R_{pqrs}&=2R_{xyzw}(
            \delta^{yp}\mathcal{I}_{ij}^{xq}R_{pqrs}
            -\delta^{yq}\mathcal{I}_{ij}^{xp}R_{pqrs}) \\
    &=2R_{xyzw}(\delta^{yp}\mathcal{I}_{ij}^{xq}R_{pqrs}
            -\delta^{yp}\mathcal{I}_{ij}^{xq}R_{qprs}) \\
    &=4\delta^{yp}\mathcal{I}_{ij}^{xq}R_{xyzw}R_{pqrs}.
\end{align*}
Similarly, we have
\begin{equation}\begin{aligned} \label{eqn Rm sharp}
    &\Rm^\#(X,Y,Z,W):=\Rm^\#(X\wedge Y,Z\wedge W) \\
    &\qquad=\frac{1}{2}X^iY^jZ^kW^l\sum_{x<y,\;z<w,\;p<q,\;r<s}
            C_{(ij)}^{(xy)(pq)}C_{(kl)}^{(zw)(rs)}R_{xyzw}R_{pqrs} \\
    &\qquad=\frac{1}{2}X^iY^jZ^kW^l\delta^{yp}\mathcal{I}_{ij}^{xq}
            \delta^{wr}\mathcal{I}_{kl}^{zs}R_{xyzw}R_{pqrs} \\
    &\qquad=\frac{1}{2}X^iY^jZ^kW^l\sum_{p,w=1}^n\bigl(R_{ipkw}R_{jplw}
            -R_{jpkw}R_{iplw} \\
    &\qquad\qquad\qquad\qquad\qquad\qquad
            -R_{iplw}R_{jpkw}+R_{jplw}R_{ipkw}\bigr) \\
    &\qquad=\sum_{p,q=1}^n\bigl(R(X,e_p,Z,e_q)R(Y,e_p,W,e_q) \\
    &\qquad\qquad\qquad-R(Y,e_p,Z,e_q)R(X,e_p,W,e_q)\bigr).
\end{aligned}\end{equation}
We remind readers that neither $\Rm^2$ nor $\Rm^\#$ is a curvature operator
        in general, but they are still well-defined as symmetric bilinear forms
        or self-adjoint linear operators and their sum is a curvature operator.

\dummyline{0.5}

Suppose $(M^n,g(t))$ is a smooth solution to the Ricci flow
\[
    \ddpfir{t}g(t)=-2\Ric_{g(t)}.
\]
We denote by $\Ric=\Ric(\Rm)$ and $\mathcal{R}=\mathcal{R}(\Rm)$ the Ricci
        curvature and scalar curvature associated to $\Rm$.
Let $\nabla_t$ denote the natural space-time extension of the Levi-Civita
        connection of $g(t)$ so that it is compatible with the metric, i.e.
\[
    \nabla_t g(t)=0.
\]
The evolution equation of the curvature operator \cite{Hamilton86JDG} is
\begin{equation} \label{eqn evol. Rm}
    \nabla_t\Rm=\Delta\Rm+2Q(\Rm)\quad\text{where}\quad Q(\Rm):=\Rm^2+\Rm^\#.
\end{equation}

\dummyline{0.5}

Let $A,B$ be symmetric bilinear forms on $\mathbb{R}^n$.
The Kulkarni-Nomizu product $A\owedge B$ is a curvature operator defined by
\[
    (A\owedge B)_{ijkl}=A_{ik}B_{jl}+A_{jl}B_{ik}-A_{il}B_{jk}-A_{jk}B_{il}.
\]
Let $\mathcal{I}$ be the identity operator on $\wedge^2\mathbb{R}^n$ and hence
        a curvature operator with constant sectional curvature $1$.

\begin{lemma}[{\cite[Lemma 2.1]{BW08AnnMath}}] \label{lma op. sharp}
    For any $\Rm\in S_B^2(\mathfrak{so}(n))$, we have
    \[
        \Rm\#\mathcal{I}=\frac{1}{2}\Ric\owedge\id-\Rm.
    \]
    In particular,
    \[
        \mathcal{I}^\#=(n-2)\mathcal{I}.
    \]
    Here we use $\Ric\owedge\id$ to denote $\Ric\owedge g$ when
            $g_{ij}=\delta_{ij}$.
\end{lemma}


\dummyline{1}

\section{Cut-off functions}

For readers' convenience, we provide a proof of existence of cut-off functions we used in this paper. 

\begin{lemma} \label{fact exist. of cutoff}
    For any $\varepsilon\in(0,1/2]$, $\sigma>0$ and $r>0$, there exists a
            smooth cut-off function $\phi: \mathbb{R}\to[0,1]$ such that
    \begin{enumerate}
        \item[\textup{(\romannumeral1)}] $\phi=1$ on $(-\infty,r]$;
        \item[\textup{(\romannumeral2)}] $\phi=0$ on $[r+\sigma,+\infty)$;
        \item[\textup{(\romannumeral3)}] $\displaystyle0\ge\phi'\ge
                -\frac{2\phi^{1-\varepsilon}}{\varepsilon\sigma}$;
        \item[\textup{(\romannumeral4)}] $\displaystyle|\phi''|\le
                \frac{5\phi^{1-2\varepsilon}}{\varepsilon^2\sigma^2}$.
    \end{enumerate}
\end{lemma}

\dummyline{0.5}

\begin{proof}
    Let $\varphi(x)$ be a smooth function such that $\varphi=1$ on $(-\infty,0]$,
            $\varphi=0$ on $[1,+\infty)$, $0\ge\varphi'\ge-2$
            and $|\varphi''|\le5$.
    This can be constructed by mollifying the piecewise linear function.
    Now we can choose the smooth function
    \[
        \phi(x):=\varphi^{\frac{1}{\varepsilon}}(\frac{x-r}{\sigma}).
    \]
    Then
    \[
        \phi'=\frac{1}{\varepsilon}\varphi^{\frac{1}{\varepsilon}-1}
                \varphi'\frac{1}{\sigma}
                =\frac{1}{\varepsilon\sigma}\phi^{1-\varepsilon}\varphi'
                \ge-\frac{2\phi^{1-\varepsilon}}{\varepsilon\sigma},
    \]
    and
    \begin{align*}
        |\phi''|&=\left|\frac{1}{\varepsilon\sigma^2}\left(\frac{1}{\varepsilon}
                -1\right)\varphi^{\frac{1}{\varepsilon}-2}(\varphi')^2
                +\frac{1}{\varepsilon\sigma^2}\varphi^{\frac{1}{\varepsilon}-1}
                \varphi''\right| \\
        &\le\frac{4|1-\varepsilon|\cdot\phi^{1-2\varepsilon}}{\varepsilon^2
                \sigma^2}+\frac{5\varepsilon\phi^\varepsilon\cdot
                \phi^{1-2\varepsilon}}{\varepsilon^2\sigma^2}.
    \end{align*}
    Therefore, the function $\phi$ satisfies all the required properties.
\end{proof}

\dummyline{1}


\bibliographystyle{amsplain}
\bibliography{88-refs}

\end{document}